\documentclass[12pt]{amsart}
\usepackage{a4wide}
\usepackage{amssymb}
\usepackage{amsthm}
\usepackage{amsmath}
\usepackage{amscd}
\usepackage{verbatim}
\usepackage[all]{xy}

\textheight8.5in \textwidth6.4in \numberwithin{equation}{section}

\theoremstyle{plain}
\newtheorem{theorem}{Theorem}[section]
\newtheorem{corollary}[theorem]{Corollary}
\newtheorem{lemma}[theorem]{Lemma}
\newtheorem{proposition}[theorem]{Proposition}

\newtheorem*{conjecture}{Conjecture}
\theoremstyle{definition}

\theoremstyle{remark}
\newtheorem{remark}{Remark}

\newcommand{\R}{\mathbb{R}}
\newcommand{\Q}{\mathbb{Q}}
\newcommand{\Z}{\mathbb{Z}}

\newcommand{\C}{\mathbb{C}}

\renewcommand{\H}{\mathbb{H}}

\newcommand{\zxz}[4]{\begin{pmatrix} #1 & #2 \\ #3 & #4 \end{pmatrix}}

\newcommand{\leg}[2]{\left( \frac{#1}{#2} \right)}
\newcommand{\kzxz}[4]{\left(\begin{smallmatrix} #1 & #2 \\ #3 & #4\end{smallmatrix}\right) }
\newcommand{\kabcd}{\kzxz{a}{b}{c}{d}}

\newcommand{\calF}{\mathcal{F}}

\newcommand{\calH}{\mathcal{H}}

\newcommand{\calQ}{\mathcal{Q}}

\newcommand{\frakc}{\mathfrak c}

\newcommand{\frake}{\mathfrak e}

\newcommand{\eps}{\varepsilon}
\newcommand{\bs}{\backslash}

\newcommand{\tr}{\operatorname{tr}}

\newcommand{\Log}{\operatorname{Log}}
\newcommand{\sgn}{\operatorname{sgn}}

\newcommand{\Sl}{\operatorname{SL}}
\newcommand{\Gl}{\operatorname{GL}}

\newcommand{\Spin}{\operatorname{Spin}}
\newcommand{\GSpin}{\operatorname{GSpin}}
\newcommand{\Mp}{\operatorname{Mp}}
\newcommand{\Orth}{\operatorname{O}}

\newcommand{\Aut}{\operatorname{Aut}}
\newcommand{\Mat}{\operatorname{Mat}}

\newcommand{\dv}{\operatorname{div}}

\newcommand{\sig}{\operatorname{sig}}
\newcommand{\res}{\operatorname{res}}

\newcommand{\Gr}{\operatorname{Gr}}

\newcommand{\SO}{\operatorname{SO}}

\newcommand{\supp}{\operatorname{supp}}

\newcommand{\Div}{\operatorname{Div}}

\begin{document}

\title[Heegner divisors, $L$-functions and
 harmonic weak Maass forms]{Heegner divisors, $L$-functions and harmonic weak Maass forms}

\date{October 1, 2007}
\author{Jan H. Bruinier and Ken Ono}

\address{Fachbereich Mathematik,
Technische Universit\"at Darmstadt, Schlossgartenstrasse 7, D--64289
Darmstadt, Germany} \email{bruinier@mathematik.tu-darmstadt.de}

\address{Department of Mathematics, University of Wisconsin,
Madison, Wisconsin 53706 USA} \email{ono@math.wisc.edu}
\thanks{The first author thanks the Max Planck Institute for Mathematics in Bonn for its support. The second author thanks the generous support of the National
Science Foundation, the Hilldale Foundation, and the Manasse family.}


\begin{abstract}
Recent works, mostly related to Ramanujan's mock theta functions,
make use of the fact that harmonic weak Maass forms can be
combinatorial generating functions. Generalizing works of
Waldspurger, Kohnen and Zagier, we prove that such forms also serve
as ``generating functions'' for central values and derivatives of
quadratic twists of weight 2 modular $L$-functions. To obtain these
results, we construct differentials of the third kind with twisted
Heegner divisor by suitably generalizing the Borcherds lift
%regularized theta lift of Borcherds, Harvey, and Moore
to harmonic
weak Maass forms. The connection with periods, Fourier coefficients,
derivatives of $L$-functions, and points in the Jacobian of modular
curves is obtained by analyzing the properties of these
differentials using works of Scholl, Waldschmidt, and Gross and
Zagier.
\end{abstract}

\maketitle

\section{Introduction and Statement of Results}\label{sect:intro}

Half-integral weight modular forms play important roles in
arithmetic geometry and number theory. Thanks to the theory of theta
functions, such forms include important generating functions for the
representation numbers of integers by quadratic forms. Among weight
3/2 modular forms, one finds Gauss' function ($q:=e^{2\pi i \tau}$
throughout)
$$
\sum_{x,y,z\in \Z}
q^{x^2+y^2+z^2}=1+6q+12q^2+8q^3+6q^4+24q^5+\cdots,
$$
which is essentially the generating function for class numbers of
imaginary quadratic fields, as well as Gross's theta functions
\cite{Gross} which enumerate the supersingular reductions of CM
elliptic curves.

In the 1980s, Waldspurger \cite{Wa}, and Kohnen and Zagier
\cite{KZ,K} established that half-integral weight modular forms also
serve as generating functions of a different type. Using the Shimura
correspondence \cite{Sh}, they proved that certain coefficients of
half-integral weight cusp forms essentially are square-roots of
central values of quadratic twists of modular $L$-functions. When
the weight is 3/2, these results appear prominently in works on the
ancient ``congruent number problem'' \cite{Tunnell}, as well as the
deep works of Gross, Zagier and Kohnen  \cite{GZ, GKZ} on the Birch
and Swinnerton-Dyer Conjecture.

In analogy with these works, Katok and Sarnak \cite{KS} employed a
Shimura correspondence to relate coefficients of weight 1/2 Maass
forms to sums of values and sums of line integrals of Maass cusp
forms. We investigate the arithmetic properties of the coefficients
of a different class of Maass forms, the weight 1/2 harmonic weak
Maass forms.

A {\it harmonic weak Maass form of weight $k\in \frac{1}{2}\Z$ on
$\Gamma_0(N)$} (with $4\mid N$ if $k\in \frac{1}{2}\Z\setminus \Z$)
is a smooth function on $\H$, the upper half of the complex plane,
which satisfies:
\begin{enumerate}
\item[(i)]
 $f\mid_k\gamma = f$ for all $\gamma\in \Gamma_0(N)$;
\item[(ii)] $\Delta_k f =0 $, where $\Delta_k$ is the weight $k$
hyperbolic Laplacian on $\H$ (see (\ref{deflap})); \item[(iii)]
There is a polynomial $P_f=\sum_{n\leq 0} c^+(n)q^n \in \C[q^{-1}]$
such that $f(\tau)-P_f(\tau) = O(e^{-\eps v})$ as $v\to\infty$ for
some $\eps>0$. Analogous conditions are required at all
cusps.
\end{enumerate}
Throughout, for $\tau\in \H$, we let $\tau=u+iv$, where $u, v\in
\R$, and we let $q:=e^{2 \pi i \tau}$.

\begin{remark}
The polynomial $P_f$, the {\it principal part of $f$ at} $\infty$,
is uniquely determined. If $P_f$ is non-constant, then $f$ has
exponential growth at the cusp $\infty$. Similar remarks apply at
all of the cusps.
\end{remark}

\begin{remark}
The results in the body of the paper are phrased in terms of vector
valued harmonic weak Maass forms. These forms are defined in
 Section \ref{sect:modularforms}.
\end{remark}

Spaces of harmonic weak Maass forms include {\it weakly holomorphic
modular forms}, those meromorphic modular forms whose poles (if any)
are supported at cusps. We are interested in those harmonic weak
Maass forms which do not arise in this way. Such forms have been a
source of recent interest  due to their connection to Ramanujan's
mock theta functions (see \cite{BO1, BO2, BOR, O2, Za3, Z1, Z2}).
For example, it
turns out that
\begin{equation}\label{fq}
M_{f}(\tau):=q^{-1}f(q^{24})+2i\sqrt{3} \cdot N_f(\tau)
\end{equation}
is a weight 1/2 harmonic weak Maass form, where
$$
N_f(\tau):=\int_{-24\overline{\tau}}^{i\infty}
\frac{\sum_{n=-\infty}^{\infty}\left(n+\frac{1}{6}\right)e^{3\pi i
\left( n+\frac{1}{6}\right)^2z}}{\sqrt{-i(z+24\tau)}}\ dz=
\frac{i}{\sqrt{3\pi}}\sum_{n\in \Z} \Gamma(1/2,4\pi (6n+1)^2 v)
q^{-(6n+1)^2}
$$
is a period integral of a theta function, $\Gamma(a,x)$ is the
incomplete Gamma function, and $f(q)$ is Ramanujan's mock theta
function
$$
f(q):=1+\sum_{n=1}^{\infty}\frac{q^{n^2}}{(1+q)^2(1+q^2)^2\cdots
(1+q^n)^2}.
$$

This example reveals two important features  common to all harmonic
weak Maass forms on $\Gamma_0(N)$.
% with weight $k\leq 1$.
Firstly, all such $f$ have Fourier expansions of the form
\begin{equation}\label{fourier}
f(\tau)=\sum_{n\gg -\infty} c^+(n) q^n + \sum_{n<0} c^-(n)W(2\pi n
v) q^n,
\end{equation}
where $W(x)=W_k(x):=\Gamma(1-k,2|x|)$.  We call $\sum_{n\gg -\infty}
c^+(n) q^n$ the {\it holomorphic part} of $f$, and we call its
complement its {\it non-holomorphic part}. Secondly, the
non-holomorphic parts  are period integrals of weight $2-k$ modular
forms. Equivalently,  $\xi_{k}(f)$ is a weight $2-k$ modular form on
$\Gamma_0(N)$, where $\xi_k$ is a differential operator (see
(\ref{defxi})) which is essentially the Maass lowering operator.

Every weight $2-k$ cusp form is the image under $\xi_{k}$ of a
weight $k$ harmonic weak Maass form. The mock theta functions
correspond  to those forms whose images under $\xi_{\frac{1}{2}}$
are weight 3/2 theta functions. We turn our attention to those
weight 1/2 harmonic weak Maass forms whose images under
$\xi_{\frac{1}{2}}$ are orthogonal to the elementary theta series.
Unlike the mock theta functions, whose holomorphic parts are often
generating functions in the theory of partitions (for example, see
\cite{BO1, BO2, BOR}), we show that these other harmonic weak Maass forms
can be ``generating functions'' simultaneously for both the values
and central derivatives of quadratic twists of weight 2 modular
$L$-functions.

Although we treat the general case in this paper, to simplify
exposition, in the introduction we assume that $p$ is prime and that
$G(\tau)=\sum_{n=1}^{\infty}B_G(n)q^n\in S_2^{new}(\Gamma_0(p))$ is
a normalized Hecke eigenform with the property that the sign of the
functional equation of
$$L(G,s)=\sum_{n=1}^{\infty}\frac{B_G(n)}{n^s}$$ is
$\epsilon(G)=-1$. Therefore, we have that $L(G,1)=0$.

By Kohnen's theory of plus-spaces \cite{K}, there is a half-integral
weight newform
\begin{equation}\label{g}
g(\tau)=\sum_{n=1}^{\infty}b_g(n)q^n\in
S_{\frac{3}{2}}^{+}(\Gamma_0(4p)),
\end{equation}
unique up to a multiplicative constant, which lifts to $G$ under the
Shimura correspondence. For convenience, we choose $g$ so that its
coefficients are in $F_G$, the totally real number field obtained by
adjoining the Fourier coefficients of $G$ to $\Q$. We shall prove
that there is a weight 1/2 harmonic weak Maass form on
$\Gamma_0(4p)$ in the plus space, say
\begin{equation}\label{fourierfg}
f_g(\tau)=\sum_{n\gg -\infty} c^+_g(n) q^n + \sum_{n<0}
c^-_g(n)W(2\pi n v) q^n,
\end{equation}
whose principal part $P_{f_g}$ has coefficients in $F_G$, which also
enjoys the property that $\xi_{\frac{1}{2}}(f_g)={\|g\|^{-2}}g$,
where $\|g\|$ denotes the usual Petersson norm.

A calculation shows that if $n>0$ (see (\ref{fourierxi})), then
\begin{equation}\label{xicoeffrelation}
  b_g(n)=-4\sqrt{\pi n} \|g\|^2\cdot
  {c_g^{-}(-n)}.
\end{equation}
The coefficients $c_g^{+}(n)$ are more mysterious. We show that both
types of coefficients are related to $L$-functions. To make this
precise, for fundamental discriminants $D$ let $\chi_D$ be the
Kronecker character for $\Q(\sqrt{D})$, and let $L(G,\chi_D,s)$ be
the quadratic twist of $L(G,s)$ by $\chi_D$. These coefficients are
related to these $L$-functions in the following way.

\begin{theorem}\label{Lvalues}
Assume that $p$ is prime, and that $G\in S_2^{new}(\Gamma_0(p))$ is
a newform. If the sign of the functional equation of $L(G,s)$ is
$\epsilon(G)=-1$, then the following are true:
\begin{enumerate}
 \item If $\Delta<0$ is a fundamental
discriminant for which $\leg{\Delta}{p}=1$, then
$$
L(G,\chi_{\Delta},1)=32\|G\|^2 \|g\|^2\pi^2 \sqrt{|\Delta|}\cdot
c^{-}_g(\Delta)^2.
$$
\item If $\Delta>0$ is a fundamental discriminant for which
$\leg{\Delta}{p}=1$, then $L'(G,\chi_{\Delta},1)=0$ if and only if
$c^{+}_g(\Delta)$ is algebraic.
\end{enumerate}
\end{theorem}

\begin{remark}  In Theorem \ref{Lvalues} (2), we have that
$L(G,\chi_{\Delta},1)=0$ since the sign of the functional equation
of $L(G,\chi_{\Delta},s)$ is $-1$. Therefore it is natural to
consider derivatives in these cases.
\end{remark}

\begin{remark} The $f_g$ are uniquely determined up to the addition of a
weight 1/2 weakly holomorphic modular form with coefficients in
$F_G$.  Furthermore, absolute values of the nonvanishing
coefficients $c_g^{+}(n)$ are typically asymptotic to subexponential
functions in $n$. For these reasons, Theorem~\ref{Lvalues} (2)
cannot be simply modified to obtain a formula for
$L'(G,\chi_{\Delta},1)$. It would be very interesting to obtain a
more precise relationship between these derivatives and the
coefficients $c_g^{+}(\Delta)$.
\end{remark}

\begin{remark} 
We give some numerical examples illustrating the theorem in Section \ref{sect:8.3}.
\end{remark}

\begin{remark} Here we comment on the construction of the weak harmonic Maass forms $f_g$.
Due to the general results in this paper, we discuss the problem in
the context of vector valued forms. It is not difficult to see that
this problem boils down to the question of producing inverse images
of classical Poincar\'e series under $\xi_{\frac{1}{2}}$. A simple
observation establishes that these preimages should be weight 1/2
Maass-Poincar\'e series which are explicitly described in Chapter 1
of \cite{Br}. Since standard Weil-type bounds fall short of
establishing convergence of these series, we briefly discuss a
method for establishing convergence. One may employ a generalization
of work of Goldfeld and Sarnak \cite{GoldSar} (for example, see
\cite{Pribitkin}) on  Kloosterman-Selberg zeta functions. This
theory proves that the relevant zeta functions are holomorphic at
$s=3/4$, the crucial point for the task at hand. One then deduces
convergence using standard methods relating the series expansions of
Kloosterman-Selberg zeta functions with their integral
representations (for example, using Perron-type formulas). The
reader may see \cite{FO} where an argument of this type is carried
out in detail.
\end{remark}

Theorem~\ref{Lvalues} (1) follows from Kohnen's theory (see
Corollary 1 on page 242 of \cite{K}) of half-integral newforms, the
existence of $f_g$, and (\ref{xicoeffrelation}). The proof of
Theorem~\ref{Lvalues} (2) is more difficult, and it involves a
detailed study of Heegner divisors. We establish that  the
algebraicity of the coefficients $c^+_g(\Delta)$ is dictated by the
vanishing of certain twisted Heegner divisors in the Jacobian of
$X_0(p)$. This result, when combined with the work of Gross and
Zagier \cite{GZ}, will imply Theorem~\ref{Lvalues} (2).

To make this precise, we first recall some definitions. Let $d<0$
and $\Delta>0$ be fundamental discriminants which are both squares
modulo $p$. Let $\calQ_{d,p}$ be the set of discriminant $d=b^2-4ac$
integral binary quadratic forms $aX^2+bXY+cY^2$ with the property
that $p\mid a$. For these pairs of discriminants, we define the {\it
twisted Heegner divisor} $Z_{\Delta}(d)$ by
\begin{equation}
Z_\Delta(d) := \sum_{Q\in \calQ_{\Delta d,p}/\Gamma_0(p)}
\chi_\Delta (Q)\cdot  \frac{\alpha_Q}{w_Q},
\end{equation}
where $\chi_\Delta$ denotes the generalized genus character
corresponding to the decomposition $\Delta\cdot d$ as in \cite{GKZ},
 $\alpha_Q$ is the unique root of $Q(x,1)$ in $\H$, and $w_Q$ denotes
 the order of the stabilizer  of $Q$ in $\Gamma_0(p)$.
Then $Z_\Delta(d)$ is a divisor on $X_0(p)$ defined over
$\Q(\sqrt{\Delta})$ (see Lemma \ref{properties}). We use these
twisted Heegner divisors to define the degree $0$ divisor
\begin{equation} y_{\Delta}(d):=Z_\Delta(d) -
\deg(Z_\Delta(d))\cdot \infty. \end{equation} Finally, we associate
a divisor to $f_g$ by letting
\begin{equation}
y_{\Delta}(f_g):=\sum_{n<0} c_g^+(n) y_{\Delta}(n)\in
\Div^0(X_0(p))\otimes F_G.
\end{equation}
Recall that we have selected $f_g$ so that the coefficients of
$P_{f_g}$ are in $F_G$.

To state our results, let $J$ be the Jacobian of $X_0(p)$, and let
$J(F)$ denote the points of $J$ over a number field $F$. The Hecke
algebra acts on $J(F)\otimes \C$, which by the Mordell-Weil Theorem
is a finite dimensional vector space. The main results of Section
\ref{sect:jac} (see Theorems~\ref{goodpoint} and  \ref{thm:alg1})
imply the following theorem.

\begin{theorem}
\label{intro:thm1} Assuming the notation and hypotheses above, the
point corresponding to $y_{\Delta}(f_g)$ in
$J(\Q(\sqrt{\Delta}))\otimes \C$ is in its $G$-isotypical component.
Moreover, the following are equivalent:
\begin{enumerate}
\item[(i)] The Heegner divisor $y_{\Delta}(f_g)$ vanishes in
$J(\Q(\sqrt{\Delta}))\otimes\C$.
\item[(ii)] The coefficient
$c^+_g(\Delta)$ is algebraic.
\item[(iii)] The coefficient
$c^+_g(\Delta)$ is contained in $F_G$.
\end{enumerate}
\end{theorem}

\noindent We then obtain the following generalization of the
Gross-Kohnen-Zagier theorem \cite{GKZ} (see Corollary
\ref{GKZ-theorem} and Theorem \ref{isogen}).
%Using certain twisted Borcherds products (see Section \ref{sect:bor}),
%Serre duality, and the Hecke action on the Jacobian we prove the following generalization
%of the Gross-Kohnen-Zagier theorem \cite{GKZ} (see Theorem \ref{isogen}).

\begin{theorem}\label{intro:thm3}
Assuming the notation and hypotheses above, we have that
\[
\sum_{n>0} y^G_\Delta(-n) q^n = g(\tau)\otimes y_\Delta(f_g) \in
S_{\frac{3}{2}}^{+}(\Gamma_0(4p))\otimes  J(\Q(\sqrt{\Delta})),
\]
where $y^G_\Delta(-n)$ denotes the projection of $y_\Delta(-n)$ onto
its $G$-isotypical component.
\end{theorem}
\noindent This result, when combined with the Gross-Zagier theorem
\cite{GZ}, gives the conclusion (see Theorem~\ref{thm:alg2}) that
the Heegner divisor $y_\Delta(f_g)$ vanishes in
$J(\Q(\sqrt{\Delta}))\otimes\C$ if and only if
$L'(G,\chi_\Delta,1)=0$, thereby proving Theorem~\ref{Lvalues} (2).

\begin{comment}
Our results arise from the interplay between Heegner divisors,
harmonic Maass forms and Borcherds products, relations which are
already of interest. In Sections~\ref{sect:lift} and \ref{sect:bor},
we extend the theory of regularized theta lifts of harmonic weak
Maass forms, and we apply these results to obtain twisted Borcherds
products. These results place harmonic weak Maass forms in the
central position which allows us to obtain the main results in this
paper.
\end{comment}

These results arise from the interplay between Heegner divisors,
harmonic weak Maass forms and Borcherds products, relations which are of independent interest. We extend the theory of regularized
theta lifts of harmonic weak Maass forms, and we apply these results
to obtain generalized Borcherds products. In that way harmonic weak
Maass forms are placed in the central position which allows us to
obtain the main results in this paper.

In view of
Theorem~\ref{Lvalues}, it is natural to investigate the algebraicity
and the nonvanishing of the coefficients of harmonic weak
Maass forms, questions which are of particular interest
in the context of elliptic curves. As a companion
to Goldfeld's famous conjecture for quadratic twists of elliptic
curves \cite{Goldfeld}, which asserts that ``half'' of the quadratic
twists of a fixed elliptic curve have rank zero (resp. one), we make
the following conjecture.

\begin{conjecture} Suppose that
$$
f(\tau)=\sum_{n\gg -\infty} c^+(n) q^n + \sum_{n<0} c^-(n)W(2\pi n
v) q^n
$$
is a weight $1/2$ harmonic weak Maass form on $\Gamma_0(N)$ whose
principal parts at cusps are defined over a number field. If
$\xi_{1/2}(f)$ is non-zero and is not a linear combination of
elementary theta series, then
\begin{displaymath}
\begin{split}
\# \{ 0< n < X \ &: \ c^{-}(-n)\neq 0\} \gg_f X,\\
\# \{ 0 < n < X \ &: \ c^{+}(n)\ {\text {is transcendental}}\} \gg_f
X.
\end{split}
\end{displaymath}
\end{conjecture}

%%\begin{remark} We believe that the
%% phenomenon in the conjecture should hold more generally
%%for harmonic weak Maass forms with other weights.
%%\end{remark}

\begin{remark}
Suppose that $G$ is as in Theorem~\ref{Lvalues}. If $G$ corresponds
to a modular elliptic curve $E$, then the truth of a sufficiently
precise form of the conjecture for $f_g$, combined with Kolyvagin's
Theorem on the Birch and Swinnerton-Dyer Conjecture, would prove
that a ``proportion'' of quadratic twists of $E$ have Mordell-Weil
rank zero (resp. one).
\end{remark}

\begin{remark}
In a recent paper, Sarnak \cite{Sarnak} characterized those Maass
cusp forms whose Fourier coefficients are all integers. He proved
that such a Maass cusp form must correspond to even irreducible
2-dimensional Galois representations which are either dihedral or
tetrahedral. More generally, a number of authors such as Langlands
\cite{langlands}, and Booker, Str\"ombergsson, and Venkatesh
\cite{Booker} have considered the algebraicity of coefficients of
Maass cusp forms. It is generally believed that the coefficients of
generic Maass cusp forms are transcendental. The conjecture above
suggests that a similar phenomenon should also hold for harmonic
weak Maass forms. In this setting, we believe that the exceptional
harmonic weak Maass forms are those which arise as preimages of
elementary theta functions such as those forms associated to the
mock theta functions.
\end{remark}

In the direction of this conjecture, we combine
Theorem~\ref{Lvalues} with works by the second author and Skinner
\cite{OSk} and Perelli and Pomykala \cite{PP} to obtain the
following result.

\begin{corollary}\label{estimates}
Assuming the notation and hypotheses above, as $X\rightarrow
+\infty$ we have
\begin{displaymath}
\begin{split}
\# \{ -X< \Delta < 0 \  {\text {fundamental}}\ &: \ c_g^{-}(\Delta)\neq 0\} \gg_{f_g} \frac{X}{\log X},\\
\# \{ 0 < \Delta < X \ {\text {fundamental}}\ &: \ c_g^{+}(\Delta)\
{\text {is transcendental}}\} \gg_{f_g,\epsilon} X^{1-\epsilon}.
\end{split}
\end{displaymath}
\end{corollary}

\begin{remark}  One can typically obtain better estimates for $c_g^{-}(\Delta)$
using properties of 2-adic Galois representations. For example, if
$L(G,s)$ is the Hasse-Weil $L$-function of an elliptic curve where
$p$ is not of the form $x^2+64$, where $x$ is an integer, then using
Theorem 1 of \cite{O} and the theory of Setzer curves \cite{Setzer},
one can find a rational number $0<\alpha<1$ for which
$$
\# \{ -X< \Delta < 0 \  {\text {fundamental}}\ : \
c_g^{-}(\Delta)\neq 0\} \gg_{f_g} \frac{X}{(\log X)^{1-\alpha}}.
$$
\end{remark}

Now we briefly provide an overview of the ideas behind the proofs of
our main
theorems. They depend on the construction of canonical
differentials of the third kind for twisted Heegner divisors. In
Section~\ref{sect:lift} we produce such differentials of the form
$\eta_{\Delta,r}(z,f)=-\frac{1}{2}\partial \Phi_{\Delta,r}(z,f)$,
where $\Phi_{\Delta,r}(z,f)$ are automorphic Green functions on
$X_0(N)$ which are obtained as liftings of weight 1/2 harmonic weak
Maass forms $f$. To define these liftings,  in Section
\ref{sect:lift} we generalize the regularized theta lift due to
Borcherds, Harvey, and Moore (for example, see \cite{Bo1},
\cite{Br}). We then employ transcendence results of Waldschmidt and
Scholl (see \cite{W}, \cite{Sch}), for the periods of differentials,
to relate the vanishing of twisted Heegner divisors in the Jacobian
to the algebraicity of the corresponding canonical differentials of
the third kind. By means of the $q$-expansion principle, we  obtain
the connection to the coefficients of harmonic weak Maass forms.

In Section~\ref{sect:bor} we construct generalized Borcherds
products for twisted Heegner divisors, and we study their properties
and multiplier systems (Theorem \ref{product}). In particular, we
give a necessary and sufficient condition that the character of such
a Borcherds product has finite order (Theorem
\ref{equivcond-borcherds}).

In Section~\ref{sect:jac}, we consider the implications of these
results when restricting to Hecke stable components. We obtain the
general versions of Theorems~\ref{Lvalues}, \ref{intro:thm1}, and
\ref{intro:thm3}, and we prove Corollary~\ref{estimates}. In
particular, Theorem~\ref{intro:thm3}, the Gross-Kohnen-Zagier
theorem for twisted Heegner divisors, is proved by adapting an
argument of Borcherds \cite{Bo2}, combined with an  analysis of the
Hecke action on cusp forms, harmonic weak Maass forms, and the
Jacobian.
%Unlike as in \cite{Bo3} we deduce
%the finiteness of the characters of Borcherds products from
%the algebraicity of the canonical differentials of the third kind.
%By means of the Gross-Zagier formula, we derive the connection to
%central critical values of twisted $L$-functions associated to $G$.

In the body of the paper we consider weight $2$ newforms $G$ of
arbitrary level and functional equation. For the regularized theta
lift, it is convenient to identify $\Sl_2$ with $\Spin(V)$ for a
certain rational quadratic space $V$ of signature $(2,1)$, and to
view the theta lift as a map from vector valued modular forms for
the metaplectic group to modular forms on $\Spin(V)$. We define the
basic setup in Section \ref{sect:prel}, and in Section
\ref{sect:difftk} we recall some results of Scholl on canonical
differentials of the third kind.  The relevant theta kernels are
then studied in Section \ref{sect:2}. They can be viewed as vector
valued weight $1/2$ versions of the weight $3/2$ twisted theta
kernels studied in \cite{Sk2} in the context of Jacobi forms. The
regularized theta lift is studied in Section \ref{sect:lift}, and
Sections \ref{sect:bor} and \ref{sect:jac} are devoted to the proofs
of the general forms of our main results as described in the
previous paragraphs.

We conclude the paper with explicit examples. In Section
\ref{sect:ex} we give examples of relations among Heegner divisors
which are {\em not} given by Borcherds lifts of weakly holomorphic
modular forms. They are obtained as lifts of harmonic weak Maass
forms. One of these examples is related to a famous example of Gross
\cite{Za1}, and the other is related to Ramanujan's mock theta
function $f(q)$. We also derive the infinite product expansions of
Zagier's twisted modular polynomials \cite{Za2}.

\section*{Acknowledgements}
We thank Fredrik Str\"omberg for his numerical computations of harmonic weak Maass forms, and we thank Tonghai Yang for many fruitful conversations.
We thank the referee(s) for their comments on a preliminary version
of this paper.

\section{Preliminaries}

\label{sect:prel}
\begin{comment}
Let $\H=\{\tau\in \C;\;\Im(\tau)>0\}$ by the complex upper half
plane. We write $\widetilde{\Gl}_2^+(\R)$ for the metaplectic
two-fold cover of $\Gl_2^+(\R)$. The elements of this group are
pairs $(M,\phi(\tau))$, where $M=\kabcd\in\Gl_2^+(\R)$ and
$\phi:\H\to \C$ is a holomorphic function with
$\phi(\tau)^2=c\tau+d$.  The multiplication is defined by
$(M,\phi(\tau)) (M',\phi'(\tau))=(M M',\phi(M'\tau)\phi'(\tau))$.
For $g=(M,\phi)\in \widetilde{\Gl}_2^+(\R)$, we put
$\det(g)=\det(M)$.
%For $M\in\Gl_2^+(\R)$ we put
%\[
%M'=\left(M,\sqrt{c\tau+d}\right)\in\widetilde{\Gl}_2^+(\R)
%\]
%with the principal branch of the holomorphic square root.
Moreover, if $G$ is a subset of $\Gl_2^+(\R)$, we write $\tilde{G}$
for its inverse image under the covering map. Throughout we write
$\Gamma=\Sl_2(\Z)$ for the full modular group. It is well known that
the integral metaplectic group $\tilde\Gamma$ is generated by $T=
\left( \kzxz{1}{1}{0}{1}, 1\right)$, and $S= \left(
\kzxz{0}{-1}{1}{0}, \sqrt{\tau}\right)$. One has the relations
$S^2=(ST)^3=Z$, where $Z=\left( \kzxz{-1}{0}{0}{-1}, i\right)$ is
the standard generator of the center of $\tilde\Gamma$.
\end{comment}

To ease exposition, the results in the introduction were stated
using the classical language of half-integral weight modular forms.
 To treat the case of general levels and functional equations, it will
be more convenient to work with vector valued forms and certain Weil
representations. Here we recall this framework, and we discuss
important theta functions which will be used to study differentials
of the third kind.

We begin by fixing notation. Let $(V,Q)$ be a non-degenerate
rational quadratic space of signature $(b^+,b^-)$. Let $L\subset V$
be an even lattice with dual $L'$. The discriminant group $L'/L$,
together with the $\Q/\Z$-valued quadratic form induced by $Q$, is
called the {\it discriminant form} of the lattice $L$.

\subsection{The Weil representation}
Let $\H=\{\tau\in \C;\;\Im(\tau)>0\}$ be the complex upper half
plane. We write $\Mp_2(\R)$ for the metaplectic two-fold cover of
$\Sl_2(\R)$. The elements of this group are pairs $(M,\phi(\tau))$,
where $M=\kabcd\in\Sl_2(\R)$ and $\phi:\H\to \C$ is a holomorphic
function with $\phi(\tau)^2=c\tau+d$.  The multiplication is defined
by
\[
(M,\phi(\tau)) (M',\phi'(\tau))=(M M',\phi(M'\tau)\phi'(\tau)).
\]
We denote the integral metaplectic group, the inverse image of
$\Gamma:=\Sl_2(\Z)$ under the covering map, by
$\tilde\Gamma:=\Mp_2(\Z)$. It is well known that $\tilde\Gamma$ is
generated by $T:= \left( \kzxz{1}{1}{0}{1}, 1\right)$, and $S:=
\left( \kzxz{0}{-1}{1}{0}, \sqrt{\tau}\right)$. One has the
relations $S^2=(ST)^3=Z$, where $Z:=\left( \kzxz{-1}{0}{0}{-1},
i\right)$ is the standard generator of the center of $\tilde\Gamma$.
We let $\tilde\Gamma_\infty:=\langle T\rangle\subset\tilde\Gamma$.

We now recall the Weil representation associated with the
discriminant form $L'/L$ (for example, see \cite{Bo1}, \cite{Br}).
It is a representation of $\tilde\Gamma$ on the group algebra
$\C[L'/L]$. We denote the standard basis elements of $\C[L'/L]$ by
$\frake_h$, $h\in L'/L$, and write $\langle\cdot,\cdot \rangle$ for
the standard scalar product (antilinear in the second entry) such
that $\langle \frake_h,\frake_{h'}\rangle =\delta_{h,h'}$. The Weil
representation $\rho_L$ associated with the discriminant form $L'/L$
is the unitary representation of $\tilde\Gamma$ on
%the group algebra
$\C[L'/L]$ defined by
\begin{align}
\label{eq:weilt}
\rho_L(T)(\frake_h)&:=e(h^2/2)\frake_h,\\
\label{eq:weils} \rho_L(S)(\frake_h)&:=
\frac{e((b^--b^+)/8)}{\sqrt{|L'/L|}} \sum_{h'\in L'/L} e(-(h,h'))
 \frake_{h'}.\\
%\end{align}
\intertext{Note that}
%\begin{align}
\label{eq:weilz} \rho_L(Z)(\frake_h)&=e((b^--b^+)/4) \frake_{-h}.
\end{align}
\begin{comment}
The orthogonal group $\Orth(L'/L)$ also acts on $\C[L'/L]$ by
\begin{align}
\label{eq:weilh} \rho_L(g)(\frake_h)&=\frake_{g^{-1}h}
\end{align}
for $g\in\Orth(L'/L)$.  The actions of $\tilde\Gamma$ and
$\Orth(L')$ commute.
\end{comment}
%This representation is essentially the Weil representation attached to
%the quadratic module $(\calL,q)$ (see \cite{No}).

\subsection{Vector valued modular forms}
\label{sect:modularforms}

If $f:\H\to \C[L'/L]$ is a function, we write $f=\sum_{\lambda\in
L'/L} f_h \frake_h$ for its decomposition in components with respect
to the standard basis of $\C[L'/L]$. Let $k\in \frac{1}{2}\Z$, and
let $M^!_{k,\rho_L}$ denote the space of $\C[L'/L]$-valued weakly
holomorphic modular forms of weight $k$ and type $\rho_L$ for the
group $\tilde\Gamma$. The subspaces of holomorphic modular forms
(resp. cusp forms) are denoted by $M_{k,\rho_L}$ (resp.
$S_{k,\rho_L}$).

%\subsection{Weak Maass forms}

Now assume that
%Let $k\in \frac{1}{2}\Z$ with
$k\leq 1$.  A twice continuously differentiable function $f:\H\to
\C[L'/L]$ is called a {\em harmonic weak Maass form} (of weight $k$
with respect to $\tilde\Gamma$ and $\rho_L$) if it satisfies:
\begin{enumerate}
\item[(i)]
$f\left(M\tau \right)=\phi(\tau)^{2k}\rho_L(M,\phi) f(\tau)$ for all
$(M,\phi)\in \tilde\Gamma$;
\item[(ii)]
there is a $C>0$ such that $f(\tau)=O(e^{C v})$ as $v\to \infty$;
% (uniformly in $u$, where $\tau=u+iv$);
%\item $f$ has polynomial growth at the cusps of $\Gamma$
%(in terms of local parameters),
%
\item[(iii)]
$\Delta_k f=0$.
\end{enumerate}
Here we have that
\begin{equation}\label{deflap}
\Delta_k := -v^2\left( \frac{\partial^2}{\partial u^2}+
\frac{\partial^2}{\partial v^2}\right) + ikv\left(
\frac{\partial}{\partial u}+i \frac{\partial}{\partial v}\right)
\end{equation}
is the usual weight $k$ hyperbolic Laplace operator, where
$\tau=u+iv$ (see \cite{BF}). We denote the vector space of these
harmonic weak Maass forms by  $\calH_{k,\rho_L}$. Moreover, we write
$H_{k,\rho_L}$ for the subspace of $f\in \calH_{k,\rho_L}$ whose
singularity at $\infty$ is locally given by the pole of a
meromorphic function. In particular, this means that  $f$ satisfies
\[
f(\tau)= P_f(\tau)+ O(e^{-\eps v}),\qquad v\to \infty,
\]
for some Fourier polynomial
\[
P_f(\tau)=\sum_{h\in L'/L}\sum_{\substack{n\in \Z+Q(h)\\-\infty\ll
n\leq 0}} c^+(n,h)e(n\tau) \frake_h
\]
and some $\eps>0$. In this situation, $P_f$ is uniquely determined
by $f$. It is called the {\em principal part} of $f$. (The space
$H_{k,\rho_L}$ was called $H^+_{k,L}$ in \cite{BF}.) We have
$M^!_{k,\rho_L}\subset H_{k,\rho_L}\subset \calH_{k,\rho_L}$. The
Fourier expansion of any $f\in H_{k,\rho_L}$ gives a unique
decomposition $f=f^++f^-$, where
\begin{subequations}
\label{deff}
\begin{align}
\label{deff+}
f^+(\tau)&= \sum_{h\in L'/L}\sum_{\substack{n\in \Q\\ n\gg-\infty}} c^+(n,h) e(n\tau)\frake_h,\\
\label{deff-} f^-(\tau)&= \sum_{h\in L'/L} \sum_{\substack{n\in \Q\\
n< 0}} c^-(n,h) W(2\pi nv) e(n\tau)\frake_h,
\end{align}
\end{subequations}
and $W(x)=W_k(x):= \int_{-2x}^\infty e^{-t}t^{-k}\,
dt=\Gamma(1-k,2|x|)$ for $x<0$.

Recall that there is an antilinear  differential operator defined by
\begin{equation}\label{defxi}
\xi_k:H_{k,\rho_L}\longrightarrow S_{2-k,\bar\rho_L},\qquad
f(\tau)\mapsto \xi_k(f)(\tau):=v^{k-2} \overline{L_k f(\tau)}.
%= R_{-k}
%v^k\overline{ f(\tau)}
\end{equation}
Here $L_k:=-2i v^2 \frac{\partial}{\partial \bar \tau}$ is the usual
Maass lowering operator. Note that $\xi_{2-k}\xi_k = \Delta_k$.
%, and $L^-$ denotes
%the lattice $L$ with the quadratic form $-Q$, so that
%$\rho_{L^-}=\bar\rho_L$.
The Fourier expansion of  $\xi_k(f)$ is given by
\begin{align}
\label{fourierxi}
\xi_k(f)= -2 \sum_{h\in L'/L} \sum_{\substack{n\in \Q\\
n>0 }} (4\pi n)^{1-k} c^-(-n,h) e(n\tau)\frake_h.
%\ \ \ \ \ {\text{\rm \texttt{check!!!!}}}-->OK
\end{align}
The kernel of $\xi_k$ is equal to $M^!_{k,\rho_L}$. By Corollary 3.8
of \cite{BF}, the following sequence is exact
\begin{gather}
\label{ex-sequ}
%\xymatrix{ 0\ar[r]& M_{k,L}^! \ar[r]& H_{k,L} \ar[r]^{\xi_k}&
%M^!_{\kappa,L^-} \ar[r] & 0,
%}\\
\xymatrix{ 0\ar[r]& M^!_{k,\rho_L} \ar[r]& H_{k,\rho_L}
\ar[r]^{\xi_k}& S_{2-k,\bar\rho_L} \ar[r] & 0 }.
\end{gather}
Moreover, by Proposition 3.11 of \cite{BF}, for any given Fourier
polynomial of the form
\[
Q(\tau)= \sum_{h\in L'/L}\sum_{\substack{n\in \Z+Q(h) \\ n<0}}
a(n,h) e(n\tau)\frake_h
\]
with $a(n,h)=(-1)^{k-\sig(L)/2}a(n,-h)$, there is an $f\in
H_{k,\rho_L}$ with principal part $P_f=Q+\frakc$ for some
$T$-invariant constant $\frakc\in \C[L'/L]$.
%such that $ f(\tau)-R(t)=O(1)$ for $v\to \infty$.

Using the Petersson scalar product, we obtain a bilinear pairing
between  $M_{2-k,\bar\rho_L}$ and $H_{k,\rho_L}$ defined by
\begin{equation}\label{defpair}
\{g,f\}=\big( g,\, \xi_k(f)\big)_{2-k} :=\int_{\Gamma\bs \H}\langle
g,\, \xi_k(f)\rangle v^{2-k}\frac{du\,dv}{v^2},
\end{equation}
where $g\in M_{2-k,\bar\rho_L}$ and $f\in H_{k,\rho_L}$. If $g$ has
the Fourier expansion $g=\sum_{h,n}b(n,h) e(n\tau)\frake_h$, and we
denote the  expansion of $f$  as in \eqref{deff}, then by
Proposition 3.5 of \cite{BF} we have
\begin{equation}\label{pairalt}
\{g,f\}= \sum_{h\in L'/L} \sum_{n\leq 0}  c^+(n,h) b(-n,h).
\end{equation}
Hence $\{g,f\}$ only depends on the principal part of $f$. The
exactness of \eqref{ex-sequ} implies that the induced pairing
between $S_{2-k,\bar\rho_L}$ and $H_{k,\rho_L}/M^!_{k,\rho_L}$ is
non-degenerate. Moreover, the pairing is compatible with the natural
$\Q$-structures on $M_{2-k,\bar\rho_L}$ and
$H_{k,\rho_L}/M^!_{k,\rho_L}$ given by modular forms with rational
coefficients and harmonic weak Maass forms with rational principal
part, respectively.
% and constant term.

We conclude this subsection with a notion which will be used later
in the paper. A harmonic weak Maass form $f\in H_{k,\rho_L}$ is said
to be {\em orthogonal to cusp forms of weight $k$} if for all $s\in
S_{k,\rho_L}$ we have
\[
(f,s)^{reg}:= \int_{\calF}^{reg} \langle f(\tau),s(\tau)\rangle
v^{k}\,\frac{du\,dv}{v^2}=0.
\]
Here $\calF$ denotes the standard fundamental domain for the action
of $\Sl_2(\Z)$ on $\H$, and the integral has been regularized as in
\cite{Bo1}.

\begin{comment}
We write $N^0_{k,\rho_L}$ (respectively $ W^0_{k,\rho_L}$) for the
subspaces of $f\in H_{k,\rho_L}$ (respectively $f\in M^!_{k,\rho_L}$)
which are orthogonal to cusp forms of weight $k$.

\begin{lemma}
If $f,f'\in N^0_{k,\rho_L}$ have the same principal parts, then
$f=f'$.
\end{lemma}

\begin{proof}
Since the principal parts of $f$ and $f'$ are equal, the difference
$f-f'$ is a cusp form in $S_{k,\rho_L}$. On the other hand, $f-f'$
must vanish since it is orthogonal to cusp forms.
\end{proof}
\end{comment}

\subsection{Siegel theta functions}

Now we recall some basic properties of theta functions associated to
indefinite quadratic forms. Let
\[
\Gr(V):=\{ z\subset V(\R): \quad \text{$z$ is a $b^+$-dimensional
subspace with $Q|_z>0$}\}
\]
be the Grassmannian of $2$-dimensional positive definite subspaces
of $V(\R)$. If $\lambda\in V(\R)$ and $z\in \Gr(V)$, we write
$\lambda_z$ and $\lambda_{z^\perp}$ for the orthogonal projection of
$\lambda$ to $z$ and $z^\perp$, respectively.

Let $\alpha,\beta\in V(\R)$. For $\tau=u+iv\in\H$ and $z\in \Gr(V)$,
we define a theta function by
\begin{align}
\label{def:vartheta} \vartheta_L(\tau, z, \alpha,\beta
):=v^{b^-/2}\sum_{\substack{\lambda\in L }} e\left(
\frac{(\lambda+\beta)_z^2}{2} \tau
+\frac{(\lambda+\beta)_{z^\perp}^2}{2}\bar\tau-(\lambda+\beta/2,\alpha)\right).
\end{align}

\begin{proposition}
\label{poisson} We have the transformation formula
\[
\vartheta_L(-1/\tau, z, -\beta,\alpha)
=\left(\frac{\tau}{i}\right)^{\frac{b^+-b^-}{2}} |L'/L|^{-1/2}
\vartheta_{L'}(\tau,z,\alpha,\beta).
\]
\end{proposition}

\begin{proof}
  This follows by Poisson summation (for example, see
  Theorem 4.1 of \cite{Bo1}).
\end{proof}

The proposition can be used to define vector valued Siegel theta
functions of weight $k=(b^+-b^-)/2$ and type $\rho_L$ (see Section 4
of \cite{Bo1}).

\subsection{A lattice related to $\Gamma_0(N)$}

\label{subsect:1.3}

Let $N$ be a positive integer. We consider the rational quadratic
space
\begin{align}
\label{defV} V:=\{X\in \Mat_2(\Q):\; \tr(X)=0\}
\end{align}
with the quadratic form $Q(X):=-N\det(X)$. The corresponding
bilinear form is given by $(X,Y)=N\tr(XY)$ for $X,Y\in V$. The
signature of $V$ is $(2,1)$.
%As in Section \ref{sect:not},
The even Clifford algebra $C^0(V)$ of $V$ can be identified with
$\Mat_2(\Q)$. The Clifford norm on $C^0(V)$ is identified with the
determinant. The group $\Gl_2(\Q)$ acts on $V$ by
\[
\gamma.X=\gamma X\gamma^{-1},\qquad \gamma\in \Gl_2(\Q),
\]
leaving the quadratic form invariant,  inducing isomorphisms of
algebraic groups over $\Q$
\[
\Gl_2\cong \GSpin(V), \qquad \Sl_2\cong \Spin(V).
\]
%We have $\GSpin(V)\cong \Gl_2(\Q)$ and $\Spin(V)\cong \Sl_2(\Q)$.
There is an isometry from $(V,Q)$ to the trace zero part of
$(C^0(V), -N\det)$.
%(This identification is already used in the
%definition of $(V,Q)$ in \eqref{defV}.)
We let $L$ be the lattice
\begin{align}
\label{latticeN}
 L:=\left\{\zxz{b}{-a/N}{c}{-b}:\quad a,b,c\in
\Z\right\}.
\end{align}
Then the dual lattice is given by
\begin{align}
\label{latticeN2}
 L':=\left\{\zxz{b/2N}{-a/N}{c}{-b/2N}:\quad
\text{$a,b,c\in \Z$} \right\}.
\end{align}
We identify $L'/L$ with $\Z/2N\Z$. Here the quadratic form on $L'/L$
is identified with the quadratic form $x\mapsto x^2$ on $\Z/2N\Z$.
The level of $L$ is $4N$.

\begin{comment}
Recall that an  integral binary quadratic form is a half-integral
symmetric matrix. We may interpret it as $1/2$ times the Gram matrix
of an even quadratic form on $\Z^2$ with respect to the standard
basis. By the discriminant of a binary quadratic form, we mean the
negative of the Gram determinant following the classical convention.
\end{comment}

If $D\in \Z$, let $L_D$ be the set of vectors $X\in L'$ with
$Q(X)=D/4N$.  Notice that $L_D$ is empty unless $D$ is a square
modulo $4N$. For $r\in L'/L$ with $r^2\equiv D \pmod{4N}$ we define
\[
L_{D,r}:=\{ X\in L':\quad\text{$Q(X)=D/4N$ and $X\equiv
r\pmod{L}$}\}.
\]
We write $L_{D}^0$ for the subset of {\em primitive} vectors in
$L_{D}$, and $L_{D,r}^0$ for the primitive vectors in $L_{D,r}$,
respectively. If $X=\kzxz{b/2N}{-a/N}{c}{-b/2N}\in L_{D,r}$, then
the matrix
\begin{equation}\label{eq:qf}
%\zxz{2A}{B}{B}{2C}:=
\psi(X):=\zxz{a}{b/2}{b/2}{Nc}=X\zxz{0}{N}{-N}{0}
\end{equation}
defines an integral binary quadratic form of discriminant
$D=b^2-4Nac=4NQ(X)$ with $b\equiv  r\pmod{2N}$.

It is easily seen that the natural homomorphism $\SO(L)\to
\Aut(L'/L)$ is surjective. Its kernel is called the discriminant
kernel subgroup, which we denote by $\Gamma(L)$. We write $\SO^+(L)$
for the intersection of $\SO(L)$ and the connected component of the
identity of $\SO(V)(\R)$.
\begin{comment}So we have the
exact sequence
\begin{align}
\label{disckernel}
 1\longrightarrow \Gamma(L)\longrightarrow
\SO(L)\longrightarrow \Orth(L'/L)\longrightarrow 1.
\end{align}
\end{comment}
The group $\Gamma_0(N)\subset \Spin(V)$ takes $L$ to itself and acts
trivially on $L'/L$.

\begin{proposition}
\label{g0n} The image of $\Gamma_0(N)$ in $\SO(L)$ is equal to
 $\Gamma(L)\cap \SO^+(L)$. The image in $\SO(L)$ of  the extension of  $\Gamma_0(N)$ by all
Atkin-Lehner involutions is equal to
 $\SO^+(L)$.
\end{proposition}
In particular, $\Gamma_0(N)$ acts on $L_{D,r}$ and $L^0_{D,r}$. By
reduction theory, the number of orbits of $L^0_{D,r}$ is finite. The
number of orbits of $L_{D,r}$ is finite if $D\neq 0$.

%\texttt{Proof in K25 notes.}

%
%\hfill $\square$
%\texttt{Check this!!!}
% We checked this.

%
\begin{comment}
The Fricke involution
\begin{align}
w_N:=\zxz{0}{-1}{N}{0}\in \GSpin(V)
\end{align}
acts by
\begin{align}
w_N.\zxz{b}{-a/N}{c}{-b} = \zxz{-b}{-c/N}{a}{b}.
\end{align}
It takes $L$ to itself and acts on $L'/L$ by multiplication with
$-1$.
%We denote by $\Gamma_0(N)^*$ the extension of $\Gamma_0(N)$ by
%the Fricke involution.
\end{comment}

%Throughout, let $r\in L'/L$, and let $D$ be a discriminant coprime
%to $N$ with $r^2\equiv D \pmod{4N}$. We fix a $\tilde r \in \Z$ with
%$\tilde r\equiv r\pmod {2N}$.

\section{Differentials of the third kind}

\label{sect:difftk}

We shall construct differentials of the third kind associated to
twisted Heegner divisors using regularized Borcherds products. We
begin by recalling some general facts concerning such differentials
\cite{Sch} and \cite{Gr}. Let $X$ be a non-singular projective curve
over $\C$ of genus $g$.
%A differential of of the
%first kind is a holomorphic $1$-form on $X$.
 A differential of the first kind on $X$
is a holomorphic $1$-form. A differential of the second kind is a
meromorphic $1$-form on $X$ whose residues all vanish. A
differential of the third kind on $X$ is a meromorphic $1$-form on
$X$ whose poles are all of first order with residues in $\Z$.
In Section \ref{sect:candiff} and the subsequent sections
we will relax the condition on the integrality of the residues.
Let
$\psi$ be a differential of the third kind on $X$ that has poles at
the points $P_j$, with residues $c_j$, and is holomorphic elsewhere.
Then the {\it residue divisor} of $\psi$ is
\[
\res(\psi):=\sum_j c_j P_j.
\]
By the residue theorem, the restriction of this divisor to any
component of $X$ has degree $0$.

Conversely, if $D=\sum_j c_j P_j$ is any divisor on $X$ whose
restriction to any component of $X$ has degree $0$, then the Riemann-Roch theorem and Serre duality imply that there is a differential $\psi_D$ of the third
kind with residue divisor $D$ (see e.g.~\cite{Gr}, p.~233). Moreover, $\psi_D$ is determined by
this condition up to addition of a differential of the first kind.
Let $U=X\setminus\{ P_j\}$. The canonical homomorphism $H_1(U,\Z)\to
H_1(X,\Z)$ is surjective and its kernel is spanned by the classes of
small circles $\delta_j$ around the points $P_j$. In particular, we
have $\int_{\delta_j} \psi_D =2\pi i c_j$.

Using the Riemann period relations, it can be shown that there is a
unique  differential  of the third kind $\eta_D$ on $X$ with residue
divisor $D$ such that
\[
\Re \left(\int_{\gamma} \eta_D\right) =0
\]
for all $\gamma\in H_1(U,\Z)$. It is called the \emph{canonical
differential of the third kind} associated with $D$. For instance,
if $f$ is a meromorphic function on $X$ with divisor $D$, then
$df/f$ is a canonical differential of the third kind on $X$ with
residue divisor $\dv(f)$. A different characterization of $\eta_D$
is given in Proposition 1 of \cite{Sch}.

\begin{proposition}
\label{dtk} The differential $\eta_D$ is the unique differential of
the third kind with residue divisor $D$ which can be written as
$\eta_D=\partial h$, where $h$ is a harmonic function on $U$.
%\hfill$\square$
\end{proposition}

Let $\bar\Q\subset\C$ be a fixed algebraic closure of $\Q$. We now
assume that the curve $X$ and the divisor $D$ are defined over a
number field $F\subset \bar\Q$. The following theorem by Scholl on the
transcendence of canonical differentials of the third kind will be
important for us (see Theorem 1 of \cite{Sch}). Its proof is based
on results by Waldschmidt on the transcendence of periods of
differentials of the third kind (see Section 5.2 of \cite{W}, and
Theorem 2 of \cite{Sch}).

\begin{theorem}[Scholl]
\label{scholl} If some non-zero multiple of $D$ is a principal
divisor, then $\eta_D$ is defined over $F$. Otherwise, $\eta_D$ is
not defined over $\bar \Q$.
%\hfill $\square$
\end{theorem}

\subsection{Differentials of the third kind on modular curves}

We consider the modular curve $Y_0(N):=\Gamma_0(N)\bs\H$. By adding
cusps in the usual way, we obtain the compact modular curve
$X_0(N)$. It is well known that $X_0(N)$ is defined over $\Q$. The
cusps are defined over $\Q(\zeta_N)$, where $\zeta_N$ denotes a
primitive $N$-th root of unity. The action of the Galois group on
them can be described explicitly (for example, see \cite{Ogg}). In
particular, it turns out that the cusps are defined over $\Q$ when
$N$ or $N/2$ is square-free. Moreover, by the Manin-Drinfeld
theorem, any divisor of degree $0$ supported on the cusps is a
multiple of a principal divisor. We let $J$ be the Jacobian
of $X_0(N)$, and let $J(F)$ denote its points over any number field
$F$. They correspond to divisor classes of degree zero on $X_0(N)$
which are rational over $F$. By the Mordell-Weil theorem, $J(F)$ is
a finitely generated abelian group.

Let $\psi$ be a differential of the third kind on $X_0(N)$. We may
write $\psi=2\pi i f \,dz$, where $f$ is a meromorphic modular form
 of weight 2 for the group $\Gamma_0(N)$.
All poles of $f$ lie on $Y_0(N)$ and are of first order, and they
have residues in $\Z$. In a neighborhood of the cusp $\infty$, the
modular form $f$ has a Fourier expansion
\[
f(z)=\sum_{n\geq 0}^\infty a(n) q^n.
\]
The constant coefficient $a(0)$ is the residue of $\psi$ at
$\infty$. We have analogous expansions at the other cusps. According
to the $q$-expansion principle, $\psi$ is defined over a number
field $F$, if and only if all Fourier coefficients $a(n)$ are
contained in $F$. Combining these facts with Theorem \ref{scholl},
we obtain the following criterion.

\begin{theorem}
\label{cd2} Let $D$ be a divisor of degree $0$ on $X_0(N)$ defined
over $F$. Let $\eta_D$ be the canonical differential of the third
kind associated to $D$ and write $\eta_D=2\pi i f \,dz$. If some
non-zero multiple of $D$ is a principal divisor, then all
the coefficients $a(n)$ of $f$ at the cusp $\infty$ are contained in
$F$. Otherwise, there exists an $n$ such that $a(n)$ is
transcendental.
\end{theorem}
%%\hfill $\square$

\section{Twisted Siegel theta functions}

\label{sect:2} To define a generalized theta lift in the next
section, we first must consider twisted Siegel theta functions. We
begin with some notation. Let $N$ be a positive integer, and let $L$
be the lattice defined in Section \ref{subsect:1.3}.
% \eqref{latticeN} of signature $(2,1)$.
Let $\Delta\in \Z$ be a fundamental discriminant and $r \in \Z$ such
that $\Delta\equiv r^2\pmod{4N}$.  Following \cite{GKZ}, we define a
generalized genus character for
$\lambda=\kzxz{b/2N}{-a/N}{c}{-b/2N}\in L'$ by putting
\[
\chi_{\Delta}(\lambda)=\chi_{\Delta}([a,b,Nc]):=
\begin{cases}
\left(\frac{\Delta}{n}\right),&\text{if $\Delta\mid b^2-4Nac$
and $(b^2-4Nac)/\Delta$ is a}\\
&\text{square modulo $4N$ and $\operatorname{gcd}(a,b,c,\Delta)=1$,}\\
0,& \text{otherwise}.
\end{cases}
\]
Here $[a,b,Nc]$ is the integral binary quadratic form corresponding
to $\lambda$, and $n$ is any integer prime to $\Delta$ represented
by one of the quadratic forms $[N_1a,b,N_2 c]$ with $N_1N_2=N$ and
$N_1,N_2>0$ (see Section 1.2 of \cite{GKZ}, and also Section 1 of
\cite{Sk2}).

The function $\chi_{\Delta}$ is invariant under the action of
$\Gamma_0(N)$ and under the action of all Atkin-Lehner involutions.
Hence it is invariant under $\SO^+(L)$.  It can be computed by the
following explicit formula (see Section I.2, Proposition 1 of
\cite{GKZ}): If $\Delta=\Delta_1\Delta_2$ is a factorization of
$\Delta$ into discriminants, and $N=N_1 N_2$ is a factorization of
$N$ into positive factors such that $(\Delta_1, N_1 a) = (\Delta_2,
N_2c)=1$, then
\begin{align}
\label{chiexplicit} \chi_{\Delta}([a,b,Nc])=\leg{\Delta_1}{N_1
a}\leg{\Delta_2}{N_2 c}.
\end{align}
If no such factorizations of $\Delta$ and $N$ exist, then we have
$\chi_{\Delta}([a,b,Nc])=0$.

We define a twisted variant of the Siegel theta function for $L$ as
follows. For a coset $h\in L'/L$, and variables $\tau=u+iv\in \H$,
$z\in \Gr(V)$, we put
\begin{align}
\label{defthetah}
\theta_{\Delta,r ,h}(\tau,z)&:=v^{1/2}\sum_{\substack{\lambda\in L+r  h \\
Q(\lambda) \equiv \Delta Q(h) \;(\Delta)}} \chi_{\Delta}(\lambda)
e\left( \frac{1}{|\Delta|} \frac{\lambda_z^2}{2} \tau
+\frac{1}{|\Delta|}\frac{\lambda_{z^\perp}^2}{2}\bar\tau\right)
\\
\nonumber &=v^{1/2}\sum_{\substack{\lambda\in L+r  h \\Q(\lambda)
\equiv \Delta Q(h) \;(\Delta)}} \chi_{\Delta}(\lambda) e\left(
\frac{1}{|\Delta|} \frac{\lambda^2}{2} u +\frac{1}{|\Delta|}\left(
\frac{\lambda_z^2}{2}-\frac{\lambda_{z^\perp}^2}{2}\right)iv\right).
\end{align}
%Here
%\[
%\tau'=\begin{cases} \tau,&\text{if $\Delta>0$,}\\
%-\bar\tau, &\text{if $\Delta<0$.}
%\end{cases}
%\]
%Notice that in the above sum the contribution by all $\lambda$ with
%$\Delta\nmid 4NQ(\lambda)$ is zero by the definition of
%$\chi_{\Delta}$.
Moreover, we define a $\C[L'/L]$-valued theta function by putting
\begin{align}
\Theta_{\Delta,r }(\tau,z):=\sum_{h\in L'/L} \theta_{\Delta,r
,h}(\tau,z)\frake_h.
\end{align}
We will often omit the dependency on $\Delta,r $ from the notation
if it is clear from the context. In the variable $z$, the function
$\Theta_{\Delta,r }(\tau,z)$ is invariant under $\Gamma_0(N)$. In
the next theorem we consider the transformation behavior in the
variable $\tau$.

\begin{theorem}
\label{twistedtheta} The theta function $\Theta_{\Delta,r }(\tau,z)$
is a non-holomorphic $\C[L'/L]$-valued modular form for $\Mp_2(\Z)$
of weight $1/2$. It transforms with the representation $\rho_L$ if
$\Delta>0$, and with $\bar\rho_L$ if $\Delta<0$.
\end{theorem}

%\texttt{Check if it really translates into Skoruppas theta
%function! ---> done.}

Theorem 4.1 of \cite{Bo1} gives the $\Delta=1$ case. For general
$\Delta$, a similar result for Jacobi forms is contained
%a formulation in terms of vector valued modular form of
%the corresponding result for Jacobi forms
in \cite{Sk2} (see \S2, pp.507). The following is crucial for its
proof.

\begin{proposition}
\label{gauss-sum} For $h\in L'/L$ and $\lambda\in L'/\Delta L$, the
exponential sum
\[
G_h(\lambda,\Delta,r )= \sum_{\substack{\lambda' \in L'/\Delta L\\
    \lambda'\equiv  r h \;(L)\\
    Q(\lambda') \equiv \Delta Q(h) \;(\Delta)}}
\chi_{\Delta}(\lambda')
e\left(-\frac{1}{|\Delta|}(\lambda,\lambda')\right)
\]
is equal to
\[
\eps |\Delta|^{3/2} \chi_{\Delta}(\lambda)
\sum_{\substack{h'\in L'/L\\ \lambda\equiv r  h'\;(L)\\
    Q(\lambda) \equiv \Delta Q(h') \;(\Delta)}} e\big(-\sgn(\Delta)(h,h')\big).
\]
Here $\eps=1$ if $\Delta>0$, and $\eps=i$ if $\Delta<0$.
\end{proposition}

\begin{proof} By applying a finite Fourier transform in $r'$ modulo
$2N$, the claim follows from identity (3) on page 517 of \cite{Sk2}.
\end{proof}

\begin{proof}[Proof of Theorem \ref{twistedtheta}]
  We only have to check the transformation behavior under the
  generators $T$ and $S$ of $\tilde\Gamma$.  The transformation law
  under $T$ follows directly from the definition in \eqref{defthetah}.
For the transformation law under $S$ we notice that we may write
\[
\theta_h(\tau,z)= \sum_{\substack{\alpha \in L'/\Delta L\\
    \alpha\equiv  r h \;(L)\\
    Q(\alpha) \equiv \Delta Q(h) \;(\Delta)}}
\chi_{\Delta}(\alpha) |\Delta|^{-1/2} \vartheta_L(|\Delta|\tau,z,0,
\alpha/|\Delta|),
\]
where $\vartheta_L$ is the theta function for the lattice $L$
defined in \eqref{def:vartheta}. Here we have used that
$\chi_{\Delta}(\lambda)$ only depends on $\lambda\in L'$ modulo
$\Delta L$. By Proposition \ref{poisson}, we find that
\begin{align*}
\theta_h(-1/\tau,z) &=\sqrt{\frac{\tau}{i}}
|L'/L|^{-1/2}|\Delta|^{-1}
 \sum_{\substack{\alpha \in L'/\Delta L\\
    \alpha\equiv  r h \;(L)\\
    Q(\alpha) \equiv \Delta Q(h) \;(\Delta)}}
\chi_{\Delta}(\alpha)
\vartheta_{L'}(\tau/|\Delta|,z, \alpha/|\Delta|,0)\\
&=\sqrt{\frac{\tau}{i}} (2N)^{-1/2}|\Delta|^{-3/2}
 v^{1/2}
\sum_{\lambda\in L'} G_h(\lambda,\Delta,r ) e\left(
\frac{1}{|\Delta|}\frac{\lambda_z^2}{2} \tau
  +\frac{1}{|\Delta|} \frac{\lambda_{z^\perp}^2}{2}\bar\tau\right).
\end{align*}
By Proposition \ref{gauss-sum}, we obtain
\[
\theta_h(-1/\tau,z) =\sqrt{\frac{\tau}{i}}\eps (2N)^{-1/2}
\sum_{h'\in L'/L} e\big(\sgn(\Delta)(h,h')\big) \theta_{h'}(\tau,z).
\]
This completes the proof of the theorem.
\end{proof}

\subsection{Partial Poisson summation}

\label{partpoisson}

We now consider the Fourier expansion of $\theta_h(\tau,z)$ in the
variable $z$. Following \cite{Bo1} and \cite{Br}, it is obtained by
applying a partial Poisson summation to the theta kernel.

Recall that the cusps of $\Gamma_0(N)$ correspond to
$\Gamma_0(N)$-classes of primitive isotropic vectors in $L$. Let
$\ell\in L$ be a primitive isotropic vector. Let $\ell' \in L'$ with
$(\ell,\ell')=1$. The $1$-dimensional lattice
\[
K=L\cap\ell'{}^\perp\cap\ell^\perp
\]
is positive definite. For simplicity we assume that $(\ell,L)=\Z$.
In this case we may chose $\ell'\in L$. Then $L$ splits into
\begin{align}
\label{eq:split} L=K\oplus\Z\ell'\oplus\Z\ell,
\end{align}
and $K'/K\cong L'/L$. (If $N$ is squarefree, then any primitive
isotropic vector $\ell \in L$ satisfies $(\ell,L)=\Z$ and our
assumption is not a restriction. For general $N$, the results of
this section still hold with the appropriate modifications, but the
formulas get considerably longer.)

We denote by $w$ the orthogonal complement of $\ell_z$ in $z$. Hence
\[
V(\R)=z\oplus z^\perp = w\oplus\R\ell_z\oplus \R \ell_{z^\perp}.
\]
If $\lambda\in V(\R)$, let $\lambda_w$ be the orthogonal projection
of $\lambda$ to $w$. We denote by $\mu$ the vector
\[
\mu=\mu(z):=-\ell'+\frac{\ell_z}{2\ell_z^2}+\frac{\ell_{z^\perp}}{2\ell_{z^\perp}^2}
\]
in $V(\R)\cap\ell^\perp$. The Grassmannian of $K$ consists of a
single point. Therefore we omit the variable $z$ in the
corresponding theta function $\vartheta_K$ defined in
\eqref{def:vartheta}.

Let $\alpha,\beta\in \Z$, and let $\mu\in K\otimes_\Z\R$. For $h\in
K'/K$ and  $\tau\in \H$,  we let
\begin{align}
\label{defxih} \xi_h(\tau,\mu, \alpha,\beta)&:= \sum_{\lambda\in
K+r h} \sum_{\substack{t\;(\Delta)\\
Q(\lambda-\beta\ell'+t\ell)\equiv \Delta Q(h)\;(\Delta)}}
\chi_{\Delta}(\lambda-\beta\ell'+t\ell) e(-\alpha
t/|\Delta|)\\
\nonumber &\phantom{=}{}\ \ \ \times e\left(
\frac{(\lambda+\beta\mu)^2}{2}\frac{\tau}{|\Delta|}-\frac{1}{|\Delta|}(\lambda+\beta\mu/2,\alpha\mu)\right).
\end{align}
Moreover, we define a $\C[K'/K]$-valued theta function by putting
\begin{align}
\Xi(\tau,\mu, \alpha,\beta):=\sum_{h\in K'/K}\xi_h(\tau,\mu,
\alpha,\beta)\frake_h.
\end{align}
Later we will use the following slightly more explicit formula for
$\Xi(\tau,\mu,n,0)$.

\begin{proposition}
\label{xispecial} If $n$ is an integer, then we have
\[
\xi_h(\tau,\mu, n,0)=   \leg{\Delta}{n}\bar\eps|\Delta|^{1/2}
\sum_{\substack{\lambda\in K+r h\\Q(\lambda)\equiv \Delta
Q(h)\;(\Delta)}} e\left(
\frac{\lambda^2}{2}\frac{\tau}{|\Delta|}-\frac{n}{|\Delta|}(\lambda,\mu)\right).
\]
Here $\leg{\Delta}{0}=1$ if $\Delta=1$, and $\leg{\Delta}{0}=0$
otherwise.
\end{proposition}

\begin{proof}
By definition we have
\[
\xi_h(\tau,\mu, n,0)= \sum_{\substack{\lambda\in
K+r h\\
Q(\lambda)\equiv \Delta Q(h)\;(\Delta)}} \sum_{t\;(\Delta)}
\chi_{\Delta}(\lambda+t\ell) e(-n t/|\Delta|) e\left(
\frac{\lambda^2}{2}\frac{\tau}{|\Delta|}-\frac{n}{|\Delta|}(\lambda,\mu)\right).
\]
Using the $\SO^+(L)$-invariance of $\chi_{\Delta}$, we find that
$\chi_{\Delta}(\lambda+t\ell)=\chi_{\Delta}([t,*,0])=\leg{\Delta}{t}$
for $\lambda\in  K+r h$ with $Q(\lambda)\equiv \Delta
Q(h)\pmod{\Delta}$. Inserting the value of the Gauss sum
\begin{align}
\label{gauss} \sum_{t\;(\Delta)} \leg{\Delta}{t}e(n t/|\Delta|)
=\leg{\Delta}{n}\eps|\Delta|^{1/2},
\end{align}
we obtain the assertion.
\end{proof}

\begin{theorem}
\label{twistedxi} If $(M,\phi)\in \Mp_2(\Z)$ with $M=\kabcd$, then
we have that
\[
\Xi(M\tau,\mu,a\alpha+b\beta,c\alpha+d\beta)=\phi(\tau)\tilde\rho_K(M,\phi)
\cdot \Xi(\tau,\mu,\alpha,\beta).
\]
Here $\tilde\rho_K$ is the representation $\rho_K$ when $\Delta>0$,
and the representation $\bar\rho_K$ when $\Delta<0$.
\end{theorem}

The proof is based on the following proposition.

\begin{proposition}
\label{gauss-sum2} Let $h\in K'/K$. For $\kappa\in K'/\Delta K$,
$a\in \Z/\Delta\Z$, and $s\in \Z/\Delta\Z$, the exponential sum
\[
g_h(\kappa,a,s)= \!\!\!\sum_{\substack{\kappa' \in K'/\Delta K\\
    \kappa'\equiv  r h \;(L)\\ b'\;(\Delta)\\
    Q(\kappa'+s \ell'+b'\ell) \equiv \Delta Q(h) \;(\Delta)}}
\!\!\!\chi_{\Delta}(\kappa'+s \ell'+b'\ell)
e\left(-\frac{1}{|\Delta|}((\kappa,\kappa') +ab')\right)
\]
is equal to
\[
\eps |\Delta|^{1/2} \sum_{\substack{h'\in K'/K\\ \kappa\equiv r
h'\;(K)}
}e\big(-\sgn(\Delta)(h,h')\big)\!\!\!\sum_{\substack{b\;(\Delta)\\
    Q(\kappa+a\ell'+b\ell) \equiv \Delta Q(h') \;(\Delta)}}
\!\!\!\chi_{\Delta}(\kappa+a\ell'+b\ell) e(bs/|\Delta|).
\]
Here $\eps=1$ if $\Delta>0$, and $\eps=i$ if $\Delta<0$.
\end{proposition}

\begin{proof}
This follows from Proposition \ref{gauss-sum} for
$\lambda=\kappa+a\ell'+b\ell$, by applying a finite Fourier
transform in $b$ modulo $\Delta$.
\end{proof}

\begin{proof}[Proof of Theorem \ref{twistedxi}]
We only have to check the transformation behavior under the
generators $T$ and $S$ of $\tilde\Gamma$.  The transformation law
under $T$ follows directly from \eqref{defxih}. For the
transformation law under $S$ we notice that we may write
\begin{align*}
\xi_h(\tau,\mu,\alpha,\beta)&=
\sum_{\substack{\lambda_1 \in K'/\Delta K\\
    \lambda_1 \equiv  r h \;(K)\\
    t\;(\Delta)\\
    Q(\lambda_1-\beta \ell' +t\ell) \equiv \Delta Q(h) \;(\Delta)}}
\chi_{\Delta}(\lambda_1-\beta\ell'+t\ell) e\left(-\frac{\alpha t
}{|\Delta|}-\frac{(\lambda_1,\alpha\mu)}{2|\Delta|}\right)\\
&\phantom{=}{}\times\vartheta_K\left(|\Delta|\tau,\alpha\mu,
\frac{\lambda_1}{|\Delta|}+\frac{\beta\mu}{|\Delta|}\right),
\end{align*}
where $\vartheta_K$ is the theta function for the lattice $K$ in
\eqref{def:vartheta}. By Proposition \ref{poisson}, we have
\begin{align*}
\xi_h(-1/\tau,\mu,-\beta,\alpha)&=\sqrt{\frac{\tau}{i}}|K'/K|^{-1/2}|\Delta|^{-1/2}
\sum_{\lambda\in K'} g_h(\lambda,-\beta,-\alpha)\\
&\phantom{=}{}\times e\left(
\frac{(\lambda+\beta\mu)^2}{2}\frac{\tau}{|\Delta|}-\frac{1}{|\Delta|}(\lambda+\beta\mu/2,\alpha\mu)\right)
.
\end{align*}
By Proposition \ref{gauss-sum2}, we find that
\begin{align*}
\xi_h(-1/\tau,\mu,-\beta,\alpha)&=\sqrt{\frac{\tau}{i}}\eps|K'/K|^{-1/2}
\sum_{h'\in K'/K}
e\big(\sgn(\Delta)(h,h')\big)\xi_h(\tau,\mu,\alpha,\beta).
\end{align*}
This concludes the proof of the theorem.
\end{proof}

\begin{lemma}
\label{thetafourier1} We have that
\begin{align*}
\theta_h(\tau,z)
&=\frac{1}{\sqrt{2|\Delta|\ell_z^2}}\sum_{\lambda\in r
h+L/\Z\ell}\sum_{d\in \Z}
\sum_{\substack{t\;(\Delta)\\Q(\lambda+t\ell)\equiv\Delta
Q(h)\;(\Delta)}}
\chi_{\Delta}(\lambda+t\ell) e(-dt/|\Delta|)\\
&\phantom{=}{}\times e\left(
\frac{\lambda_w^2}{2}\frac{\tau}{|\Delta|}-\frac{d}{|\Delta|}\frac{(\lambda,\ell_z-\ell_{z^\perp})}{2\ell_z^2}-\frac{|d+(\lambda,\ell)\tau|^2}{4i|\Delta|v\ell_z^2}\right).
\end{align*}
\end{lemma}

\begin{proof}
The proof follows the argument of Lemma 5.1 in \cite{Bo1} (see also
Lemma 2.3 in \cite{Br}). In the definition of $\theta_h(\tau,z)$, we
rewrite the sum over $\lambda\in rh+L$ as a sum over
$\lambda'+d|\Delta|\ell$, where $\lambda'$ runs through
$rh+L/\Z\Delta\ell$ and $d$ runs through $\Z$. We obtain
\[
\theta_{h}(\tau,z)=v^{1/2}\sum_{\substack{\lambda\in r  h+
L/\Z\Delta\ell\\ Q(\lambda) \equiv \Delta Q(h) \;(\Delta)}}
\chi_{\Delta}(\lambda)\sum_{d\in \Z}
g(|\Delta|\tau,z,\lambda/|\Delta|;d),
\]
where the function $g(\tau,z,\lambda;d)$ is defined by
\[
g(\tau,z,\lambda;d)= e\left( \frac{(\lambda+d\ell)_z^2}{2} \tau
+\frac{(\lambda+d\ell)_{z^\perp}^2}{2}\bar\tau\right)
\]
for $\tau\in\H$, $z\in \Gr(V)$, $\lambda\in V(\R)$, and $d\in \R$.
We apply Poisson summation to the sum over $d$. The Fourier
transform of $g$ as a function in $d$ is
\[
\hat g(\tau,z,\lambda,d)=\frac{1}{\sqrt{2v\ell_z^2}} e\left(
\frac{\lambda_w^2}{2}\tau
-\frac{d(\lambda,\ell_z-\ell_{z^\perp})}{2\ell_z^2}-\frac{|d+(\lambda,\ell)\tau|^2}{4iv\ell_z^2}\right)
\]
(see \cite{Br}, p.~43). We obtain
\begin{align*}
\theta_{h}(\tau,z)&=v^{1/2}\sum_{\substack{\lambda\in r  h+
L/\Z\Delta\ell\\ Q(\lambda) \equiv \Delta Q(h) \;(\Delta)}}
\chi_{\Delta}(\lambda)\sum_{d\in \Z} \hat
g(|\Delta|\tau,z,\lambda/|\Delta|;d)\\
&=\frac{1}{\sqrt{2|\Delta|\ell_z^2}}\sum_{\substack{\lambda\in r  h+ L/\Z\Delta\ell\\
Q(\lambda) \equiv \Delta Q(h) \;(\Delta)}}
\chi_{\Delta}(\lambda)\sum_{d\in \Z} e\left(
\frac{\lambda_w^2}{2}\frac{\tau}{|\Delta|}
-\frac{d(\lambda,\ell_z-\ell_{z^\perp})}{2|\Delta|\ell_z^2}-\frac{|d+(\lambda,\ell)\tau|^2}{4iv|\Delta|\ell_z^2}\right).
\end{align*}
The claim follows by rewriting the sum over $\lambda\in
rh+L/\Z\Delta\ell$ as a sum over $\lambda'+t\ell$, where $\lambda'$
runs through $rh+L/\Z\ell$ and $t$ runs through $\Z/\Delta \Z$, and
by using the facts that $\ell_w=0$ and
$(\ell,\ell_z-\ell_{z^\perp})/2\ell_z^2=1$.
\end{proof}

\begin{lemma}
\label{thetafourier2} We have that
\begin{align*}
\theta_h(\tau,z) &=\frac{1}{\sqrt{2|\Delta|\ell_z^2}}\sum_{c,d\in
\Z} \exp\left(-\frac{\pi
|c\tau+d|^2}{2|\Delta|v\ell_z^2}\right)\xi_h(\tau,\mu,d,-c).
\end{align*}
\end{lemma}

\begin{proof}
Using $rh+L/\Z\ell=rh+K+\Z\ell'$ and the identities
\begin{align*}
\ell'_w&=-\mu_w,\\
-\frac{\mu^2}{2}&=\frac{(\ell',\ell_z-\ell_{z^\perp})}{2\ell_z^2},\\
(\lambda,\mu)&=\frac{(\lambda,\ell_z-\ell_{z^\perp})}{2\ell_z^2}
\end{align*}
for $\lambda\in K\otimes\R$, the formula of Lemma
\ref{thetafourier1} can rewritten as
\begin{align*}
\theta_h(\tau,z)
&=\frac{1}{\sqrt{2|\Delta|\ell_z^2}}\sum_{\lambda\in r
h+K}\sum_{c,d\in \Z}
\sum_{\substack{t\;(\Delta)\\Q(\lambda+c\ell'+t\ell)\equiv\Delta
Q(h)\;(\Delta)}}
\chi_{\Delta}(\lambda+c\ell'+t\ell) e(-dt/|\Delta|)\\
&\phantom{=}{}\times e\left(
\frac{(\lambda-c\mu)_w^2}{2}\frac{\tau}{|\Delta|}-\frac{1}{|\Delta|}
(\lambda-c\mu/2,d\mu)-\frac{|c\tau+d|^2}{4i|\Delta|v\ell_z^2}\right).
\end{align*}
Inserting the definition \eqref{defxih} of $\xi_h(\tau,\mu,
\alpha,\beta)$, we obtain the assertion.
\end{proof}

\begin{theorem}
\label{thetafourier3} We have that
\begin{align*}
\Theta_{\Delta,r }(\tau,z)
&=\frac{1}{\sqrt{2|\Delta|\ell_z^2}}\Xi(\tau,0,0,0)\\
&\phantom{=}{}+\frac{1}{\sqrt{2|\Delta|\ell_z^2}} \sum_{n\geq
1}\sum_{\gamma\in \tilde\Gamma_\infty\bs\tilde\Gamma}
\left[\exp\left(-\frac{\pi n^2}{2|\Delta|\Im(
\tau)\ell_z^2}\right)\Xi(\tau,\mu,n,0)\right]\mid_{1/2,\tilde\rho_K}\gamma.
\end{align*}
\end{theorem}

\begin{proof}
According to Lemma \ref{thetafourier2}, we have
\begin{align*}
\Theta_{\Delta,r }(\tau,z) &=
\frac{1}{\sqrt{2|\Delta|\ell_z^2}}\sum_{c,d\in \Z}
\exp\left(-\frac{\pi
|c\tau+d|^2}{2|\Delta|v\ell_z^2}\right)\Xi(\tau,\mu,d,-c)\\
&=\frac{1}{\sqrt{2|\Delta|\ell_z^2}}\Xi(\tau,\mu,0,0)\\
&\phantom{=}{} +\frac{1}{\sqrt{2|\Delta|\ell_z^2}}\sum_{n\geq
1}\sum_{\substack{c,d\in \Z\\ (c,d)=1}}\exp\left(-\frac{\pi n^2
|c\tau+d|^2}{2|\Delta|v\ell_z^2}\right)\Xi(\tau,\mu,nd,-nc).
\end{align*}
Writing the sum over coprime integers $c,d$ as a sum over
$\tilde\Gamma_\infty\bs \tilde\Gamma$ and using the transformation
law for $\Xi(\tau,\mu,\alpha,\beta)$ of Theorem \ref{twistedxi}, we
obtain the assertion.
\end{proof}

According to Proposition \ref{xispecial}, the function
$\Xi(\tau,0,0,0)$ in Theorem \ref{thetafourier3} vanishes when
$\Delta\neq 1$. When $\Delta=1$ we have
\[
\Xi(\tau,0,0,0)=\sum_{\lambda\in K'}
e(Q(\lambda)\tau)\frake_\lambda,
\]
and so $\Xi(\tau,0,0,0)$ is the usual vector valued holomorphic
theta function of the one-dimensional positive definite lattice $K$.

Let $k_0$ be a basis vector for $K$. If $y\in K\otimes\R$, we write
$y>0$ if $y$ is a positive multiple of $k_0$. Let $f\in
H_{1/2,\tilde\rho_L}$. We define the Weyl vector corresponding to
$f$ and $\ell$ to be the unique $\rho_{f,\ell}\in K'\otimes\R$  such
that
\begin{align}
\label{weylvector}
(\rho_{f,\ell},y)=\frac{\sqrt{(y,y)}}{8\pi\sqrt{2|\Delta|}}\int_\calF^{reg}
\langle f(\tau),\Xi(\tau,0,0,0)\rangle v^{1/2}\frac{du\,dv}{v^{2}}
\end{align}
for all $y\in K\otimes \R$ with $y>0$. Here $\calF$ denotes the
standard fundamental domain for the action of $\Sl_2(\Z)$ on $\H$,
and the integral has to be regularized as in \cite{Bo1}. We have,
$\rho_{f,\ell}=0$ when $\Delta\neq 1$. (This is also true for cusps
given by primitive isotropic vectors $\ell$ with $(\ell,L)\neq \Z$.)
One can show that $\rho_{f,\ell}$ does not depend on the choice of
the vector $\ell'$. The sign of $\rho_{f,\ell}$ depends on the
choice of $k_0$.

%The following facts on Weyl vectors will be important for us.

We conclude this section with an important fact on the rationality
of Weyl vectors $\rho_{f,\ell}$.

\begin{proposition}
\label{prop:weylvectorrational} Let $f\in H_{1/2,\tilde\rho_L}$ be a
harmonic weak Maass form with coefficients $c^\pm(m,h)$ as in
\eqref{deff}. If $c^+(m,h)\in \Q$ for all $m\leq 0$ and
 $f$ is orthogonal to weight $1/2$ cusp forms,
then  $\rho_{f,\ell}\in \Q$.
\end{proposition}

\begin{proof}
The idea of the proof is similar to \S9 of \cite{Bo1}. But notice
that Lemma 9.5 of \cite{Bo1} is actually only true if $N$ is a
prime. Therefore we need some additional care. Since
$\rho_{f,\ell}=0$ when $\Delta\neq 1$, we only need to consider the
case $\Delta=1$. Let $E_{3/2}(\tau)$ be the weight $3/2$ Eisenstein
series for $\tilde \Gamma$ with representation $\bar \rho_L$
%(defined by holomorphic continuation as usual)
normalized to have constant term $\frake_0$. It turns out that
\[
\xi_{3/2}(E_{3/2})(\tau)=
C\frac{\sqrt{N}}{16\pi}\Xi(\tau,0,0,0)+s(\tau),
\]
where $C$ is a non-zero rational constant and $s\in S_{1/2,\rho_L}$
is a cusp form. Hence the integral in \eqref{weylvector} can be
computed by means of \eqref{pairalt} in terms of the coefficients of
the holomorphic part of the Eisenstein series $E_{3/2}$. These
coefficients are known to be generalized class numbers and thereby
rational. Since $f$ is orthogonal to cusp forms, the cusp form $s$
does not give any contribution to the integral. This concludes the
proof of the proposition.
\end{proof}

\begin{remark} Proposition~\ref{prop:weylvectorrational} does not
hold without  the hypothesis that $f$ is orthogonal to weight $1/2$
cusp forms.
\end{remark}

\begin{comment}
If $N$ is square-free this follows from Proposition \ref{xispecial},
which implies that the Weyl vector corresponding to $f$ and any cusp
$\ell$ vanishes. For general $N$ it follows by slightly generalizing
the setting of Section~\ref{partpoisson} to arbitrary primitive
isotropic vectors $\ell\in L$ as in Section~5 of \cite{Bo1} or
Chapter~2.1 of \cite{Br}. For simplicity we did not include this
straightforward generalization in the present paper.

We let the Weyl vector corresponding to $f$ (and $\ell$) be the
unique $\rho_f\in K'\otimes\R$  such that
\begin{align}
\label{weylvector} (\rho_f,y)=\frac{1}{8\pi\sqrt{2|\Delta|\ell_z^2}}
\int_\calF^{reg} \langle f(\tau),\Xi(\tau,0,0,0)\rangle
v^{1/2}\frac{du\,dv}{v^{2}}.
\end{align}
In view of Proposition \ref{xispecial}, $\rho_f$ vanishes unless
$\Delta=1$. The integral can essentially be computed using the
argument of \cite{Bo1} (see Section 9, Lemma 9.5 and Corollary 9.6),
taking into account the necessary correction of the construction of
the function $G_N$ given in Remark \texttt{Put in suitable
reference} of the present paper. When $f$ is weakly holomorphic and
has rational Fourier coefficients, one can show that $\rho_f\in
K'\otimes\Q$.
\end{comment}

\section{Regularized theta lifts of weak Maass forms}

\label{sect:lift}

In this section we generalize the regularized theta lift of
Borcherds, Harvey, and Moore in two ways to construct automorphic
forms on modular curves. First, we work with the twisted Siegel
theta functions of the previous section as kernel functions, and
secondly, we consider the lift for harmonic weak Maass forms.

Such generalizations have been studied previously in other settings.
In \cite{Ka}, Kawai constructed twisted theta lifts of weakly
holomorphic modular forms in a different way. However, his
automorphic products  are of higher level, and only twists by {\em
even} real Dirichlet characters are considered. In \cite{Br} and
\cite{BF}, the (untwisted) regularized theta lift was studied on
harmonic weak Maass forms and was used to construct automorphic
Green functions and harmonic square integrable representatives for
the Chern classes of Heegner divisors. However, the Chern class
construction only leads to non-trivial information about Heegner
divisors if the modular variety under consideration has dimension
$\geq 2$ (i.e. for $\Orth(2,n)$ with $n\geq 2$).   Here we consider
the $\Orth(2,1)$-case of modular curves.  The Chern class of a
divisor on a curve is just its degree, and does not contain much
arithmetic information. Hence, the approach of \cite{Br} and
\cite{BF} to study Heegner divisors does not apply.

Instead of using automorphic Green functions to construct the Chern
classes of Heegner divisors, we employ them to construct canonical
differentials of the third kind associated to twisted Heegner
divisors. By the results of Waldschmidt and Scholl of
Section~\ref{sect:difftk}, such differentials carry valuable
arithmetic information. We begin by fixing some notation.

As in Section \ref{sect:2}, let $\Delta$ be a fundamental
discriminant and let $r \in \Z$ such that $\Delta\equiv
r^2\pmod{4N}$.
We let $\ell,\ell'\in L$ be the isotropic vectors
\[
\ell =\zxz{0}{1/N}{0}{0},\qquad \ell'=\zxz{0}{0}{1}{0}.
\]
Then we have $K=\Z\kzxz{1}{0}{0}{-1}$. For $\lambda \in K\otimes\R$,
we write $\lambda>0$ if $\lambda$ is a positive multiple of
$\kzxz{1}{0}{0}{-1}$. Following Section 13 of \cite{Bo1}, and
Section 3.2 of \cite{Br}, we identify the complex upper half plane
$\H$ with an open subset of $K\otimes\C$ by mapping $t\in \H$ to
$\kzxz{1}{0}{0}{-1}\otimes t$. Moreover, we identify $\H$ with the
Grassmannian $\Gr(V)$ by mapping $t\in \H$ to the positive definite
subspace
\[
z(t)=\R \Re\zxz{t}{-t^2}{1}{-t} +\R \Im\zxz{t}{-t^2}{1}{-t}
\]
of $V(\R)$. Under this identification,  the action of $\Spin(V)$ on
$\H$ by fractional linear transformations corresponds to the linear
action on $\Gr(V)$ through $\SO(V)$. We have that
\begin{align*}
\ell_z^2&=\frac{1}{2N\Im(t)^2},\\
(\lambda,\mu)&=(\lambda,\Re(t)),\\
\lambda^2/\ell_z^2&=(\lambda,\Im(t))^2,
\end{align*}
for $\lambda\in K\otimes\R$. In the following we will frequently
identify $t$ and $z(t)$ and simply write $z$ for this variable. We
let $z=x+iy$ be the decomposition into real and imaginary part.

We now define twisted Heegner divisors on the modular curve
$X_0(N)$. For any vector $\lambda\in L'$ of negative norm, the
orthogonal complement $\lambda^\perp\subset V(\R)$ defines a point
$Z(\lambda)$ in $\Gr(V)\cong \H$. For $h\in L'/L$ and a negative
rational number $m\in \Z+\sgn(\Delta)Q(h)$, we consider the twisted
Heegner divisor
\begin{align}
Z_{\Delta,r }(m,h):= \sum_{\substack{\lambda\in L_{d\Delta,h r
}/\Gamma_0(N)}} \frac{\chi_{\Delta}(\lambda)}{w(\lambda)}
Z(\lambda)\in\Div(X_0(N))_\Q,
\end{align}
where $d:=4Nm\sgn(\Delta)\in \Z$. Note that $d$ is a discriminant
which is congruent to a square modulo $4N$ and which has the
opposite sign as $\Delta$. Here $w(\lambda)$ is the order of the
stabilizer of $\lambda$ in $\Gamma_0(N)$. (So $w(\lambda)\in
\{2,4,6\}$,  and $w(\lambda)=2$ when $d\Delta<-4$.) We also consider
the  degree zero divisor
\begin{align}
y_{\Delta,r}(m,h):=Z_{\Delta,r }(m,h)-\deg(Z_{\Delta,r
}(m,h))\cdot\infty.
%\in \Div(X_0(N))_\Q.
\end{align}
We have $y_{\Delta,r}(f)=Z_{\Delta,r}(f)$ when $\Delta\neq 1$.
%Note
%that  $D:=d\Delta$ is a negative discriminant congruent to a square
%modulo $4N$.
%
% for all negative rational numbers $m\in
%\Z+\sgn(\Delta)Q(h)$.
%
%The divisor $Z_{\Delta,r }(h,m)$ is invariant
%under $\Gamma_0(N)$ and descends to a divisor on the modular curve
%$Y_0(N)=\Gamma_0(N)\bs \H$.
By the theory of complex multiplication, the divisor $Z_{\Delta,r
}(h,m)$ is defined over $\Q(\sqrt{D},\sqrt{\Delta})$ (for example,
see \S12 of \cite{Gro}). The following lemma shows that it is
defined over $\Q(\sqrt{\Delta})$ and summarizes some further
properties.

\begin{lemma}
\label{properties} Let $w_N$ be the Fricke involution on $X_0(N)$,
and let $\tau$ denote complex conjugation, and let $\sigma$ be the
non-trivial automorphism of
$\Q(\sqrt{D},\sqrt{\Delta})/\Q(\sqrt{D})$. Then the following are
true:
\begin{enumerate}
\item[(i)]
$w_N (Z_{\Delta,r }(m,h))=Z_{\Delta,r }(m,-h)$,
\item[(ii)]
$\tau (Z_{\Delta,r }(m,h))=Z_{\Delta,r }(m,-h)$,
\item[(iii)]
$\sigma (Z_{\Delta,r }(m,h))=-Z_{\Delta,r }(m,h)$,
\item[(iv)]
$Z_{\Delta,r }(m,-h)=\sgn(\Delta)Z_{\Delta,r }(m,h)$,
\item[(v)]
$Z_{\Delta,r}(m,h)$ is defined over $\Q(\sqrt{\Delta})$.
\end{enumerate}
\end{lemma}

\begin{proof}
Properties (i) and (ii) are verified by a straightforward
computation, and (iv) immediately follows from the definition of the
genus character $\chi_\Delta$. Moreover, (iii) follows from the
theory of complex multiplication (see p. 15 of \cite{BS}, and
\cite{Gro}). Finally, (v) is a consequence of (ii), (iii), (iv).
\end{proof}

\begin{remark}
Our definition of Heegner divisors differs slightly from \cite{GKZ}.
They consider the orthogonal complements of vectors $\lambda
=\kzxz{b/2N}{-a/N}{c}{-b/2N}\in L'$ of negative norm with $a>0$, or
equivalently, zeros in the upper half plane of {\em positive
definite} binary quadratic forms.
\end{remark}

%We now consider a regularized theta lift for harmonic weak Maass forms
%using the twisted Siegel theta functions of the
%previous section.

Recall that $\tilde\rho_L=\rho_L$ for $\Delta>0$, and
$\tilde\rho_L=\bar\rho_L$ for $\Delta<0$. Let $f\in
H_{1/2,\tilde\rho_L}$ be a harmonic weak Maass form of weight $1/2$
with representation $\tilde \rho_L$.
% for $\tilde\Gamma$ with representation $\tilde\rho_L$.
%(Recall
%that $\tilde\rho_L=\rho_L$ when $\Delta$ is positive, and
%$\tilde\rho_L=\bar\rho_L$ when $\Delta$ is negative.)
We denote the coefficients of $f=f^++f^-$ by $c^\pm(m,h)$ as in
\eqref{deff}. Note that $c^\pm(m,h)=0$ unless $m\in
\Z+\sgn(\Delta)Q(h)$. Moreover, by means of \eqref{eq:weilz} we see
that $c^\pm(m,h)=c^\pm(m,-h)$ if $\Delta>0$, and
$c^\pm(m,h)=-c^\pm(m,-h)$ if $\Delta<0$. Throughout we assume that
$c^+(m,h)\in \R$ for all $m$ and $h$.
%Moreover, we suppose that the constant term
%$c^+(0,0)$ of $f$ vanishes when $\Delta=1$.

Using the Fourier coefficients of the principal part of $f$, we
define the twisted Heegner divisor associated to $f$ by
\begin{align}
Z_{\Delta,r}(f) &:=
%\frac{1}{2}
\sum_{h\in L'/L}\sum_{m<0}
c^+(m,h)Z_{\Delta,r }(m,h)\in \Div(X_0(N))_\R,\\
y_{\Delta,r}(f) &:=
%\frac{1}{2}
\sum_{h\in L'/L}\sum_{m<0} c^+(m,h)y_{\Delta,r }(m,h)\in
\Div(X_0(N))_\R.
\end{align}
Notice that $y_{\Delta,r}(f)=Z_{\Delta,r}(f)$ when $\Delta\neq 1$.
The divisors lie in $\Div(X_0(N))_\Q$ if the coefficients of the
principal part of $f$ are rational.

We define a regularized theta integral of $f$ by
\begin{align}
\label{defphi} \Phi_{\Delta,r }(z,f) = \int_{\tau\in \calF}^{reg}
\langle f(\tau),\Theta_{\Delta,r }(\tau,z)\rangle
v^{1/2}\,\frac{du\,dv}{v^2}.
\end{align}
Here $\calF$ denotes the standard fundamental domain for the action
of $\Sl_2(\Z)$ on $\H$, and the integral has to be regularized as in
\cite{Bo1}.

\begin{proposition}
  \label{prop:sing}  The theta integral $\Phi_{\Delta,r }(z,f)$ defines
  a $\Gamma_0(N)$-invariant function on $\H\bs
Z_{\Delta,r}(f)$ with a logarithmic singularity\footnote{If $X$ is a
normal
    complex space, $D\subset X$ a Cartier divisor, and $f$ a smooth
    function on $X\setminus\supp(D)$, then $f$ has a logarithmic
    singularity along $D$, if for any local equation $g$ for $D$ on an
    open subset $U\subset X$, the function $f-\log|g|$ is smooth on
    $U$.}
on the divisor $-4 Z_{\Delta,r}(f)$. If $\Omega$ denotes the
invariant Laplace operator on $\H$, we have
\[
\Omega \Phi_{\Delta,r }(z,f) = \leg{\Delta}{0} c^+(0,0).
\]
\end{proposition}

\begin{proof}
Using the argument of Section 6 of \cite{Bo1}, one can show that
$\Phi_{\Delta,r }(z,f)$ defines
  a $\Gamma_0(N)$-invariant function on $\H\bs Z_{\Delta,r}(f)$
with a logarithmic singularity on $-4 Z_{\Delta,r}(f)$. To prove the
claim concerning the Laplacian, one may argue as in Theorem 4.6 of
\cite{Br}.
\end{proof}

\begin{remark}
\label{rem:realan}
  The fact that the function $\Phi_{\Delta,r }(z,f)$ is subharmonic
  implies that it is real analytic on $\H\bs Z_{\Delta,r}(f)$ by a
  standard regularity theorem for elliptic differential operators.
\end{remark}

We now describe the Fourier expansion of $\Phi_{\Delta,r }(z,f)$.
Recall the definition of the Weyl vector $\rho_{f,\ell}$
corresponding to $f$ and $\ell$, see \eqref{weylvector}.

\begin{theorem}
\label{thm:green} For $z\in \H$ with $y\gg 0$, we have
\begin{align}
\label{phi:exp}
\Phi_{\Delta,r }(z,f)&=-4\sum_{\substack{\lambda\in K'\\
\lambda>0}}\sum_{b\;(\Delta)}
\leg{\Delta}{b}c^+(|\Delta|\lambda^2/2,r
\lambda) \log\left|1-e((\lambda,z)+b/\Delta)\right|\\
\nonumber &\phantom{=}{}+ \begin{cases}
2\sqrt{\Delta}c^+(0,0)L(1,\chi_{\Delta}),&\text{if $\Delta\neq 1$,}\\[.5ex]
8\pi(\rho_{f,\ell},y)-c^+(0,0)(\log(4\pi Ny^2)+\Gamma'(1)),
&\text{if $\Delta=1$.}
\end{cases}
\end{align}
\end{theorem}

\begin{proof}
Here we carry out the proof only in the case $\Delta\neq 1$, for
which the regularization is slightly easier and there is no Weyl
vector term. We note that when $\Delta=1$, the proof is similar, and
when $f$ is weakly holomorphic it is contained in Theorem 13.3 of
\cite{Bo1}. In our proof we essentially follow the argument of
Theorem 7.1 of \cite{Bo1}, and Theorem 2.15 of \cite{Br}. In
particular, all questions regarding convergence can be treated
analogously. Inserting the formula of Theorem \ref{thetafourier3} in
definition \eqref{defphi} and unfolding, we obtain
\begin{align}
\label{eq:h1} \Phi_{\Delta,r }(z,f)&=
%\frac{1}{\sqrt{2 |\Delta|\ell_z^2}}
%\int_\calF^{reg} \langle f(\tau),\Xi(\tau,0,0,0)\rangle
%v^{1/2}\frac{du\,dv}{v^{2}}\\
%\nonumber
%&\phantom{=}{}+
\frac{\sqrt{2}}{\sqrt{|\Delta|\ell_z^2}} \sum_{n\geq 1}
\int_{v=0}^\infty\int_{u=0}^1 \exp\left(-\frac{\pi
n^2}{2|\Delta|v\ell_z^2}\right)\langle
f(\tau),\Xi(\tau,\mu,n,0)\rangle \,du\,\frac{dv}{v^{3/2}}.
\end{align}
Here we have also used the fact that $\rho_{f,\ell}=0$ when
$\Delta\neq 1$. We temporarily denote the Fourier expansion of $f$
by
\[
f(\tau)=\sum_{h\in L'/L}\sum_{n\in \Q} c(n,h,v)e(n \tau).
\]
Inserting the formula for $\Xi(\tau,\mu,n,0)$ of Proposition
\ref{xispecial} in \eqref{eq:h1}, and carrying out the integration
over $u$, we obtain
\begin{align}
\label{fgeneral} \Phi_{\Delta,r}(z,f)&=\frac{\sqrt{2}\eps}{|\ell_z|}
\sum_{h\in K'/K}\sum_{\substack{\lambda\in K+r h\\Q(\lambda)\equiv
\Delta Q(h)\;(\Delta)}} \sum_{n\geq 1} \leg{\Delta}{n}
e\left(\frac{n}{|\Delta|}(\lambda,\mu)\right)\\
\nonumber &\phantom{=}{}\times \int_{v=0}^\infty
c(Q(\lambda)/|\Delta|,h,v) \exp\left(-\frac{\pi
n^2}{2|\Delta|v\ell_z^2}-\frac{2\pi\lambda^2v}{|\Delta|}\right)\,\frac{dv}{v^{3/2}}.
\end{align}
Since $\Delta$ is fundamental, the conditions
$Q(\lambda)\equiv\Delta Q(h)\pmod{\Delta}$ and $\lambda\equiv rh
\pmod{K}$ are equivalent to $\lambda=\Delta\lambda'$ and
$r\lambda'\equiv h\pmod{K}$ for some $\lambda'\in K'$. Consequently,
we have
\begin{align}
\label{fgeneral2}
\Phi_{\Delta,r}(z,f)&=\frac{\sqrt{2}\eps}{|\ell_z|}
\sum_{\substack{\lambda\in K'}} \sum_{n\geq 1} \leg{\Delta}{n}
e\left(\sgn(\Delta)n(\lambda,\mu)\right)\\
\nonumber &\phantom{=}{}\times \int_{v=0}^\infty
c(|\Delta|\lambda^2/2,r\lambda,v) \exp\left(-\frac{\pi
n^2}{2|\Delta|v\ell_z^2}-2\pi\lambda^2|\Delta|v\right)\,\frac{dv}{v^{3/2}}.
\end{align}
Notice that only the coefficients
$c(|\Delta|\lambda^2/2,r\lambda,v)$ where $\lambda\in K'$ occur in
the latter formula. Since $K$ is positive definite, the quantity
$|\Delta|\lambda^2/2$ is non-negative, and so  we have
\[
 c(|\Delta|\lambda^2/2,r\lambda,v)=c^+(|\Delta|\lambda^2/2,r\lambda),
\]
that is, only the coefficients of the ``holomorphic part'' $f^+$ of
$f$ give a contribution. We now compute the integral over $v$ (for
example, using page 77 of ~\cite{Br}). We obtain
\[
\int_{v=0}^\infty \exp\left(-\frac{\pi
n^2}{2|\Delta|v\ell_z^2}-2\pi\lambda^2|\Delta|v\right)\,\frac{dv}{v^{3/2}}
=\frac{\sqrt{2|\Delta|\ell_z^2}}{n}\exp(-2\pi n|\lambda|/|\ell_z|).
\]
Inserting this and separating the contribution of $\lambda=0$, we
get
\begin{align*}
\Phi_{\Delta,r}(z,f)&=2\sqrt{\Delta} c^+(0,0) \sum_{n\geq 1}
\leg{\Delta}{n}\frac{1}{n}\\
&\phantom{=}{}+ 4 \sum_{\substack{\lambda\in K'\\
\lambda>0}} c^+(|\Delta|\lambda^2/2,r\lambda)
\Re\left(\sqrt{\Delta}\sum_{n\geq 1} \frac{1}{n}\leg{\Delta}{n}
e\big(\sgn(\Delta)n(\lambda,\mu)+in|\lambda|/|\ell_z|\big)\right).
\end{align*}
Using the value of the Gauss sum \eqref{gauss}, we see that this is
equal to
\begin{align*}
\Phi_{\Delta,r}(z,f)&=2\sqrt{\Delta} c^+(0,0) L(1,\chi_\Delta)\\
&\phantom{=}{}- 4 \sum_{\substack{\lambda\in K'\\
\lambda>0}} \sum_{b\;(\Delta)} \leg{\Delta}{b}
c^+(|\Delta|\lambda^2/2,r\lambda) \log\left|
1-e\left(\frac{b}{\Delta}+(\lambda,\mu)+i\frac{|\lambda|}{|\ell_z|}\right)\right|.
\end{align*}
We finally put in the identities $(\lambda,\mu)=(\lambda,x)$ and
$|\lambda|/|\ell_z|=|(\lambda,y)|$, to derive the theorem.
\end{proof}

\begin{remark}
i) Note that for lattices of signature $(2,n)$ with $n\geq 2$, the
lattice $K$ is Lorentzian, and one gets a non-trivial contribution
from $f^-$ to the theta integral, which is investigated in
\cite{Br}. So the above situation is very special.

ii) At the other cusps of $X_0(N)$, the function
$\Phi_{\Delta,r}(z,f)$ has similar Fourier expansions as in
\eqref{phi:exp}.

iii) The function $\Phi_{\Delta,r }(z,f)$ is a Green function for
the divisor $Z_{\Delta,r}(f)+C_{\Delta,r}(f)$ in the sense of
\cite{BKK}, \cite{BBK}. Here $C_{\Delta,r}(f)$ is a divisor on
$X_0(N)$ supported at the cusps, see also \eqref{reseta}
\end{remark}

\subsection{Canonical differentials of the third kind for
Heegner divisors} \label{sect:candiff}

For the rest of this section, we assume that $f\in
H_{1/2,\tilde\rho_L}$ and that the coefficients $c^+(m,h)$ are
rational for all $m\leq 0$ and $h\in L'/L$. Moreover,  we assume
that the constant term $c^+(0,0)$ of $f$ vanishes when $\Delta=1$,
so that $\Phi_{\Delta,r}(z,f)$ is harmonic. We identify $\Z$ with
$K'$ by mapping $n\in \Z$ to $\frac{n}{2N}\kzxz{1}{0}{0}{-1}$. Then
%(assuming $c(0,0)=0$ when $\Delta=1$)
the Fourier expansion of $\Phi_{\Delta,r}(z,f)$ given in Theorem
\ref{thm:green} becomes
\begin{align}
\label{fouriernice}
\Phi_{\Delta,r}(z,f)&=2\sqrt{\Delta}c^+(0,0)L(1,\chi_{\Delta})+8\pi\rho_{f,\ell}
y\\
\nonumber &\phantom{=}{} -4\sum_{\substack{n\geq 1}}
\sum_{b\;(\Delta)}
\leg{\Delta}{b}c^+(\tfrac{|\Delta|n^2}{4N},\tfrac{rn}{2N})
\log\left|1-e(nz+b/\Delta)\right|.
\end{align}
It follows from
%Proposition  \ref{dtk}  and
Proposition \ref{prop:sing}
%Theorem \ref{green}
that
\begin{align}
\label{def:eta} \eta_{\Delta,r
}(z,f)&=-\frac{1}{2}\partial\Phi_{\Delta,r }(z,f)
\end{align}
is a differential of the third kind on $X_0(N)$.  It has the residue
divisor
\begin{align}
\label{reseta} \res(\eta_{\Delta,r
}(z,f))=Z_{\Delta,r}(f)+C_{\Delta,r}(f),
%\in\Div^0(X_0(N))_\Q,
\end{align}
where $Z_{\Delta,r}(f)\in \Div(X_0(N))_\Q$, and $C_{\Delta,r}(f)\in
\Div(X_0(N))_\R$ is a divisor on $X_0(N)$ which is supported at the
cusps. (Here we have relaxed the condition that the residues be
integral, and only require them to be real.)
%, and $C_{\Delta,r}(f)=0$ when $\Delta\neq 1$.
The multiplicity of any cusp $\ell$ in the divisor $C_{\Delta,r}(f)$
is given by the Weyl vector $\rho_{f,\ell}$. According to
Proposition \ref{prop:weylvectorrational}, if $f$ is orthogonal to
the cusp forms in $S_{1/2,\tilde \rho_L}$ then $\rho_{f,\ell}$ is
rational. When $\Delta \neq 1$, all Weyl vectors vanish and
consequently $C_{\Delta,r}(f)=0$.

%
\begin{comment}
When $\Delta=1$, the order of $\Psi_{\Delta,r}(z,f)$ at any cusp
$\ell$ of $X_0(N)$ is given by the Weyl vector corresponding to $f$
and $\ell$. In this case the divisor of $\Psi_{\Delta,r}(z,f)$ as a
function on $X_0(N)$ is given by $Z_{\Delta,r}(f)+C_{\Delta,r}(f)$,
where $C_{\Delta,r}(f)$ is a divisor supported at the cusps.
\end{comment}

\begin{comment}
Observe that $Z_{\Delta,r}(f)+C_{\Delta,r}(f)$ differs from
$y_{\Delta,r}(f)$ only by a divisor of degree $0$ supported at the
cusps.
%of a rational function on $X_0(N)$.
%\texttt{Show $C_{\Delta,r}=0$ when $\Delta\neq 1$ (?).}
In particular, by the Manin-Drinfeld theorem, the divisors
$Z_{\Delta,r}(f)+C_{\Delta,r}(f)$ and $y_{\Delta,r}(f)$ define the
same point in $J(\Q)\otimes \R$.
\end{comment}

\begin{theorem}
\label{cdfourier}
  The differential $\eta_{\Delta,r }(z,f)$ is the canonical
  differential of the third kind corresponding to
  $Z_{\Delta,r}(f)+C_{\Delta,r}(f)$.  It has the Fourier expansion
\begin{align*}
\eta_{\Delta,r }(z,f)
%&=-\frac{1}{2}\partial\Phi_{\Delta,r }(z,f)\\
&= \bigg(\rho_{f,\ell}-
\sgn(\Delta)\sqrt{\Delta}\sum_{\substack{n\geq 1}}\sum_{d\mid n}
\frac{n}{d}\leg{\Delta}{d}
c^+(\tfrac{|\Delta|n^2}{4Nd^2},\tfrac{rn}{2Nd}) e(nz)\bigg)\cdot
2\pi i\,dz.
\end{align*}
\end{theorem}

%\texttt{Be careful with cusps when $\Delta=1$!}

\begin{proof}
%The statement of Theorem \ref{green2} on the singularity of
%$\Phi_{\Delta,r}(z,f)$ implies that $\eta_{\Delta,r }(z,f)$ is a
%differential of the third kind with residue divisor $Z_{\Delta,r}(f)$.
Since $\Phi_{\Delta,r}(z,f)$ is harmonic on $\H\bs Z_{\Delta,r}(f)$,
Proposition \ref{dtk} implies that $\eta_{\Delta,r }(z,f)$ is the
canonical differential of the third kind associated with
$Z_{\Delta,r}(f)+C_{\Delta,r}(f)$. Differentiating
\eqref{fouriernice}, we obtain
\begin{align*}
\eta_{\Delta,r }(z,f)&= \bigg(\rho_{f,\ell} -\sum_{\substack{n\geq
1}}\sum_{d\geq 1} \sum_{b\;(\Delta)}
\leg{\Delta}{b}c^+(\tfrac{|\Delta|n^2}{4N},\tfrac{rn}{2N})n
e(ndz+bd/\Delta)\bigg)\cdot 2\pi i\,dz.
\end{align*}
Inserting the value of the Gauss sum \eqref{gauss} and reordering
the summation, we get the claimed Fourier expansion.
\end{proof}

\begin{comment}
% Old proof
For $\gamma=\kabcd\in \Gl_2^+(\R)$ we put as usual
\[
(F\mid_k \gamma)(z) = \det(\gamma)^{k/2} (cz+d)^{-k} F(\gamma z).
\]
Recall that the Fricke involution $w_N=\kzxz{0}{-1}{N}{0}$ takes $F$
to the modular form $G:=F\mid_k w_N$ for $\Gamma_0(N)$. We denote
its Fourier expansion by $G=\sum_{n}b(n)q^n$.

For a positive integer $m$ the matrix $V_m=\kzxz{m}{0}{0}{1}$
defines the shift operator, taking $F$ to $F\mid_k V_m$, which is a
modular form of weight $k$ for $\Gamma_0(Nm)$. It has the Fourier
expansion $F\mid_k V_m= m^{k/2}\sum_{n}a(n)q^{mn}$.

We consider the modular form
\[
H:=F-m^{-k/2}F\mid_k V_m = \sum_{ n} (a(n)-a(n/m))q^n
\]
for the group $\Gamma_0(Nm)$. In view of the hypothesis it has
algebraic Fourier coefficients. Therefore, by the $q$-expansion
principle $H\mid_k W_{Nm}$ has algebraic Fourier coefficients as
well. We have
\begin{align*}
H\mid_k W_{Nm} &= (F-m^{-k/2}F\mid_k V_m)\mid_k W_{Nm}\\
&=(F\mid_k W_{N}\mid_k V_m -m^{-k/2}F\mid_k W_{N}\mid_k \zxz{m}{0}{0}{m}\\
&=m^{k/2} G(mz)-m^{-k/2}G(mz).
\end{align*}
Hence, the quantities $m^k b(n/m)-b(n)$ are algebraic for all
integers $n$ (and all positive integers $m$). Because $k\neq 0$ we
may conclude that $b(n)$ is algebraic for all $n\in \Z$, that is,
$G$ is defined over a number field. But since $F=G\mid w_N$, this
implies that $F$ is also defined over a number field, and in
particular $a(0)$ is algebraic.
\end{proof}
\end{comment}

\begin{theorem}
\label{equivcond} Assume that $\Delta\neq 1$. The following are
equivalent.
\begin{enumerate}
\item[(i)] A non-zero multiple of  $y_{\Delta,r}(f)$ is the divisor of a rational function on $X_0(N)$.
\item[(ii)] The coefficients $c^+(\tfrac{|\Delta|n^2}{4N},\tfrac{rn}{2N})$ of $f$  are algebraic for all positive integers $n$.
\item[(iii)] The coefficients $c^+(\tfrac{|\Delta|n^2}{4N},\tfrac{rn}{2N})$ of $f$ are rational for all positive integers $n$.
\end{enumerate}
\end{theorem}

\begin{proof}
  Statement (iii) trivially implies (ii).
  If (ii) holds, then, in
  view of Theorem~\ref{cdfourier},
the canonical differential
  $\eta_{\Delta,r}(z,f)$ of the divisor
  $y_{\Delta,r}(f)\in \Div(X_0(N))_\Q$ is defined over $\bar\Q$.
  Consequently, Theorem \ref{scholl} implies that a non-zero multiple of
  $y_{\Delta,r}(f)$ is the divisor of a rational
  function on $X_0(N)$. Hence (i) holds.

  It remains to prove that
  (i) implies (iii).
If $y_{\Delta,r}(z,f)$ is a non-zero multiple of the divisor of a
rational function on $X_0(N)$, then Lemma \ref{properties} and
Theorem \ref{scholl} imply that $\eta_{\Delta,r}(z,f)$ is defined
over $F=\Q(\sqrt{\Delta})$, the field of definition of
$y_{\Delta,r}(f)$. Using the $q$-expansion principle and M\"obius
inversion, we deduce from Theorem~\ref{cdfourier}, for every
positive integer $n$, that
\begin{align}
\label{etacoeff} \sqrt{\Delta} n
c^+(\tfrac{|\Delta|n^2}{4N},\tfrac{rn}{2N})\in F.
\end{align}
%If $\Delta=1$ we obtain (iv). If $\Delta\neq 1$, then
Denote by $\sigma$ the non-trivial automorphism of $F/\Q$. It
follows from Lemma \ref{properties} that
$\sigma(y_{\Delta,r}(f))=-y_{\Delta,r}(f)$. Hence
$\sigma(\eta_{\Delta,r}(z,f))=-\eta_{\Delta,r}(z,f)$. Using the
action of $\sigma$ on the $q$-expansion of $\eta_{\Delta,r}(z,f)$,
we find that $\sigma$ fixes the coefficients
$c^+(\tfrac{|\Delta|n^2}{4N},\tfrac{rn}{2N})$. Consequently, these
coefficients are rational.
\end{proof}

\begin{remark}
\label{orthcusp1} Theorem \ref{equivcond} also holds for $\Delta=1$
when $S_{1/2,\tilde\rho_L}=0$, or more generally when $f$ is
orthogonal to the cusp forms in $S_{1/2,\tilde\rho_L}$. The latter
conditions ensure that the Weyl vectors corresponding to $f$ are
rational and thereby $C_{\Delta,r}(f)\in \Div((X_0(N))_\Q$. Observe
that $Z_{\Delta,r}(f)+C_{\Delta,r}(f)$ differs from
$y_{\Delta,r}(f)$ only by a divisor of degree $0$ supported at the
cusps.
%of a rational function on $X_0(N)$.
%\texttt{Show $C_{\Delta,r}=0$ when $\Delta\neq 1$ (?).}
In particular, by the Manin-Drinfeld theorem, the divisors
$Z_{\Delta,r}(f)+C_{\Delta,r}(f)$ and $y_{\Delta,r}(f)$ define the
same point in $J(\Q)\otimes \R$.\end{remark}

\begin{remark}
Assume that the equivalent conditions of Theorem \ref{equivcond}
hold.

i) It is interesting to consider whether the rational coefficients
$c^+(\tfrac{|\Delta|n^2}{4N},\tfrac{rn}{2N})$ in (iii) have bounded
denominators. This is true if $f$ is weakly holomorphic, since
$M^!_{1/2,\tilde\rho_L}$ has a basis of modular forms with integral
coefficients. However, if $f$ is an honest harmonic weak Maass form,
it is not clear at all.

ii) The rational function in Theorem \ref{equivcond} (i) has an
automorphic product expansion as in Theorem \ref{product}. It is
given by a non-zero power of $\Psi_{\Delta,r}(z,f)$.
\end{remark}

\section{Generalized Borcherds products}
\label{sect:bor}

In this section we consider certain  automorphic products which
arise as liftings of harmonic weak Maass forms and which can be
viewed as generalizations of the automorphic products  in Theorem
13.3 of \cite{Bo1}. In particular, for any Heegner divisor
$Z_{\Delta,r}(m,h)$, we obtain a meromorphic automorphic product
$\Psi$ whose divisor on $X_0(N)$ is the sum of $Z_{\Delta,r}(m,h)$
and a divisor supported at the cusps. But unlike the results in
\cite{Bo1}, the function $\Psi$ will in general transform with a
multiplier system of infinite order under $\Gamma_0(N)$. We then
give a criterion when the multiplier system has finite order.

As usual, for complex numbers $a$ and $b$, we let $a^b=\exp(b
\Log(a))$, where $\Log$ denotes the principal branch of the complex
logarithm. In particular, if $|a|<1$ we have
$(1-a)^b=\exp(-b\sum_{n\geq 1} \frac{a^n}{n})$.

%Assume that the coefficients $c^+(m,h)$ are rational for all $m\in \Q$ and all $h\in L'/L$.

\begin{theorem}
\label{product} Let $f\in H_{1/2,\tilde\rho_L}$ be a harmonic weak
Maass form with real coefficients $c^+(m,h)$ for all $m\in \Q$ and
$h\in L'/L$.
%$c^+(|\Delta|\lambda^2/2,r \lambda)$ for all $\lambda \in K'$.
Moreover, assume that $c^+(n,h)\in \Z$ for all $n\leq 0$. The
infinite product
\[
\Psi_{\Delta,r }(z,f)= e((\rho_{f,\ell},z))\prod_{\substack{\lambda\in K'\\
\lambda>0}}\prod_{b\;(\Delta)}\left[1-e((\lambda,z)+b/\Delta)
\right]^{\leg{\Delta}{b}c^+(|\Delta|\lambda^2/2,r \lambda)}
\]
converges for $y$ sufficiently large and has a meromorphic
continuation to all of $\H$ with the following properties.
\begin{enumerate}
\item[(i)] It is a meromorphic modular form for
$\Gamma_0(N)$ with a unitary character $\sigma$ which may have
infinite order.
\item[(ii)] The weight of $\Psi_{\Delta,r }(z,f)$ is $c^+(0,0)$ when $\Delta=1$, and is $0$ when $\Delta\neq 1$.
\item[(iii)]
The divisor of $\Psi_{\Delta,r }(z,f)$ on $X_0(N)$
%=\Gamma_0(N)\bs\H$
is given by $Z_{\Delta,r}(f)+C_{\Delta,r}(f)$.
%:=\frac{1}{2}\sum_{h\in L'/L}\sum_{m<0}
%c(m,h)Z_{\Delta,r }(m,h).
%\]
\item[(iv)]
We have
\[
\Phi_{\Delta,r }(z,f)=\begin{cases} -c^+(0,0)(\log(4\pi
N)+\Gamma'(1))-4
\log|\Psi_{\Delta,r }(z,f)y^{c^+(0,0)/2}|,& \text{if $\Delta=1$,}\\[.5ex]
2\sqrt{\Delta}c(0,0)L(1,\chi_{\Delta})-4 \log|\Psi_{\Delta,r
}(z,f)|,& \text{if $\Delta\neq 1$.}
\end{cases}
\]
\end{enumerate}
\end{theorem}

\begin{proof}
By means of the same argument as in Section 13 of \cite{Bo1}, or
Chapter 3 of \cite{Br}, it can be deduced from Proposition
\ref{prop:sing}, Remark \ref{rem:realan}, and Theorem
\ref{thm:green} that $\Psi_{\Delta,r }(z,f)$ has a continuation to a
meromorphic function on $\H$ satisfying (iii) and (iv). Moreover,
using the $\Gamma_0(N)$-invariance of $\Phi_{\Delta,r }(z,f)$ one
finds that it satisfies the transformation law
\[
\Psi_{\Delta,r }(\gamma z,f)=\sigma(\gamma)(cz+d)^{c^+(0,0)}
\Psi_{\Delta,r }(z,f),
\]
for $\gamma=\kabcd\in \Gamma_0(N)$, where $\sigma :\Gamma_0(N)\to
\C^\times$ is a unitary  character of $\Gamma_0(N)$.
\end{proof}

\begin{theorem}
\label{equivcond-borcherds} Suppose that $\Delta\neq 1$. Let $f\in
H_{1/2,\tilde\rho_L}$ be a harmonic weak Maass form with real
coefficients $c^+(m,h)$ for all $m\in \Q$ and  $h\in L'/L$.
Moreover, assume that $c^+(n,h)\in \Z$ for all $n\leq 0$. The
following are equivalent.
\begin{enumerate}
\item[(i)]  The character $\sigma$ of the function $\Psi_{\Delta,r }(z,f)$ defined in Theorem \ref{product} is of finite order.
\item[(ii)] The coefficients $c^+(|\Delta|\lambda^2/2,r \lambda)$ are rational for all $\lambda \in K'$.
\end{enumerate}
\end{theorem}

\begin{proof}
%We first prove the theorem for $\Delta\neq 1$.
If (i) holds, then there is a positive integer $M$ such that
$\Psi_{\Delta,r }(z,f)^M$ is a rational function on $X_0(N)$ with
divisor $M \cdot Z_{\Delta,r}(f)$. By means of Theorem
\ref{equivcond} we find that $c^+(|\Delta|\lambda^2/2,r \lambda)\in
\Q$ for all $\lambda \in K'$. Conversely, if (ii) holds, using
Theorem \ref{equivcond} we may conclude that
 $ M \cdot Z_{\Delta,r}(f)$ is the divisor of
a rational function $R$ on $X_0(N)$ for some positive integer $M$.
But this implies that
\[
\log|R|- M\log|\Psi_{\Delta,r }(z,f)|
\]
is a harmonic function on $X_0(N)$ (without any singularities). By
the maximum principle, it is constant. Hence $R/\Psi_{\Delta,r
}(z,f)^M$ is a holomorphic function on $\H$ with constant modulus,
which must be constant. Consequently, $\sigma^M$ is the trivial
character.
%
\begin{comment}
Secondly, suppose that $\Delta=1$. By adding a suitable integral
multiple of the vector valued weight $1/2$ theta function for the
lattice $K$ we may assume that the constant term $c^+(0,0)$ of $f$
vanishes. (Notice that the automorphic product corresponding to this
theta function is the eta product $\eta(z)\eta(Nz)$.) Now we may
argue as in the first case.
\end{comment}
\end{proof}

\begin{remark}
\label{orthcusp2} Theorem \ref{equivcond-borcherds} also holds for
$\Delta=1$ when $S_{1/2,\tilde\rho_L}=0$, or more generally when $f$
is orthogonal to the cusp forms in $S_{1/2,\tilde\rho_L}$. The
latter conditions ensure that  Theorem~\ref{equivcond} still applies
(see Remark \ref{orthcusp1}). Notice that we may reduce to the case
that the constant term $c^+(0,0)$ of $f$ vanishes by adding a
suitable rational linear combination of $\Orth(K'/K)$-translates of
the vector valued weight $1/2$ theta series for the lattice $K$.
\end{remark}

The rationality of the coefficients $c^+(m,h)$ of a harmonic weak
Maass form is usually not easy to verify. In view of Theorem
\ref{equivcond-borcherds}, it is related to the vanishing of twisted
Heegner divisors in the Jacobian, which is a deep question (see
\cite{GZ}). However for special harmonic weak Maass forms, such as
the mock theta functions, one can read off the rationality directly
from the construction. This leads to explicit relations among
certain Heegner divisors on $X_0(N)$ comparable to the relations
among cuspidal divisors coming from modular units.

For weakly holomorphic modular forms $f\in M^!_{1/2,\tilde\rho_L}$ the
rationality of the Fourier coefficients is essentially dictated by
the principal part (with minor complications caused by the presence
of cusp forms of weight $1/2$). More precisely, one can use the
following lemmas.

%\texttt{Define $y_{\Delta,r}(f)$ and $M^!_{1/2,\rho}(K)$ already here.}

\begin{comment}
Hence it remains to show that $\sigma$ is of finite order. For
$\Delta=1$ Borcherds proves this in \cite{Bo3} using the embedding
trick. Since this argument does not generalize to $\Delta\neq 1$ and
to liftings of harmonic weak Maass forms in an obvious way, we use a
different argument here.
\end{comment}

\begin{lemma}
\label{rationalcoeff} Suppose that $f\in M^!_{1/2,\tilde\rho_L}$. If
$c^+(m,h)\in \Q $ for all $m\leq 0$, then there exists a cusp form
$f'\in S_{1/2,\tilde\rho_L}$ such that $f+f'$ has rational
coefficients.
\end{lemma}

\begin{proof}
This follows from the fact that the spaces $M^!_{1/2,\tilde\rho_L}$
and $S_{1/2,\tilde\rho_L}$ have bases of modular forms with rational
coefficients (see \cite{McG}).
\end{proof}

% \texttt{See also K25 notes.}

\begin{remark}
Observe that $M_{1/2,\bar\rho_L}$ is always trivial, by a result of
Skoruppa (see Theorem 5.7 of \cite{EZ}). However, $M_{1/2,\rho_L}$
may contain non-zero elements (which are linear combinations of
theta series of weight $1/2$).
\end{remark}

\begin{lemma}
\label{weakrat} If $f\in S_{1/2,\rho_L}$, then the coefficient
$c^+(m,h)$ of $f$ vanishes unless $m=\lambda^2/2$ for some
$\lambda\in K'$.
\end{lemma}

\begin{proof}
This could be proved using the Serre-Stark basis theorem. Here we
give a more indirect proof using the twisted theta lifts of the
previous section.

Let  $\Delta\neq 1$ be a fundamental discriminant and let $r$ be an
integer satisfying $\Delta\equiv r^2\pmod{4N}$. We consider the
canonical differential of the third kind $\eta_{\Delta,r}(z,f)$.
Since $f$ is a cusp form and $\Delta\neq 1$, the divisor
$Z_{\Delta,r}(f)$ vanishes. Consequently,
$\eta_{\Delta,r}(z,f)\equiv 0$. By Theorem \ref{cdfourier}, we find
that $c^+(|\Delta|\lambda^2/2,r\lambda)=0$ for all $\lambda\in K'$.
This proves the lemma.
\end{proof}

\begin{comment}
\begin{lemma}
\label{rationalcoeff1} Let $f\in W^0_{1/2,\tilde\rho}$ and assume
that $c^+(m,h)\in \Q $ for all $m\leq 0$. Then $c^+(m,h)\in \Q$ for
all $m$.
\end{lemma}

\begin{proof}
\texttt{Put in !}
\end{proof}
\end{comment}

\begin{lemma}
\label{rationalcoeff2} If $f\in M^!_{1/2,\tilde\rho_L}$, then the
following are true:
\begin{enumerate}
\item[(i)]
If $c^+(m,h)=0$ for all $m<1$, then $f$ vanishes identically.
\item[(ii)] If
$c^+(m,h)\in \Q$ for all $m<1$, then all coefficients of $f$ are
contained in $\Q$.
\end{enumerate}
\end{lemma}

\begin{proof}
(i) The hypothesis implies that $f/\Delta$ is a holomorphic modular
form of weight $-23/2$ with representation $\tilde\rho_L$. Hence it
has to vanish identically. (ii) The assertion follows from (i) using
the Galois action on $ M^!_{1/2,\tilde\rho_L}$.
\end{proof}

\begin{remark}
% i) Using Lemma \ref{properties} and the fact that
% $c(m,-h)=\sgn(\Delta)c(m,h)$,
%  it is easily seen that the divisor $
% Z_{\Delta,r}(f)$ is always defined over $\Q(\sqrt{\Delta})$.
%
i) If $f\in M^!_{1/2,\tilde\rho_L}$ is a weakly holomorphic modular
form with rational coefficients $c^+(m,h)$ and $c^+(m,h)\in \Z$ for
$m\leq 0$, then Theorem \ref{equivcond-borcherds} and Remark
\ref{orthcusp2} show that the automorphic product
$\Psi_{\Delta,r}(z,f)$ of Theorem \ref{product} is a meromorphic
modular form for $\Gamma_0(N)$ with a character of finite order.
When $\Delta=1$, this result is contained in Theorem 13.3 of
\cite{Bo1}. Borcherds proved the finiteness of the multiplier system
in \cite{Bo3} using the embedding trick. (However, the embedding
trick argument does not work for harmonic weak Maass forms.) In the
special case that $N=1$ and $\Delta>0$, twisted Borcherds products
were first constructed by Zagier in a different way (see \S7 of
\cite{Za2}).

ii) The space of weakly holomorphic modular forms of weight $1/2$
with representation $\rho_L$ is isomorphic to the space of weakly
skew holomorphic Jacobi forms in the sense of \cite{Sk1}. For these
forms, Theorem \ref{product} gives the automorphic products
$\Psi_{\Delta,r}$ for any {\em positive} fundamental discriminant
$\Delta$. The space of weakly holomorphic modular forms of weight
$1/2$ with representation $\bar\rho_L$ is isomorphic to the space of
``classical'' weakly holomorphic Jacobi forms in the sense of
\cite{EZ}. For these forms, the theorem gives the automorphic
products $\Psi_{\Delta,r}$ for any {\em negative} fundamental
discriminant $\Delta$.
\begin{comment}
iii) When $\Delta \neq 1$, the order of $\Psi_{\Delta,r}(z,f)$ at
any cusp of $X_0(N)$ is zero. Hence the divisor of
$\Psi_{\Delta,r}(z,f)$ as a function on $X_0(N)$ is given by
$Z_{\Delta,r}(f)$. When $N$ is square-free this follows from
Proposition \ref{xispecial}, which implies that the Weyl vector
corresponding to $f$ and any cusp $\ell$ vanishes. For general $N$
it follows by slightly generalizing the setting of
Section~\ref{partpoisson} to arbitrary primitive isotropic vectors
$\ell\in L$ as in Section~5 of \cite{Bo1} or Chapter~2.1 of
\cite{Br}. For simplicity we did not include this straightforward
generalization in the present paper.
%
When $\Delta=1$, the order of $\Psi_{\Delta,r}(z,f)$ at any cusp
$\ell$ of $X_0(N)$ is given by the Weyl vector corresponding to $f$
and $\ell$. In this case the divisor of $\Psi_{\Delta,r}(z,f)$ as a
function on $X_0(N)$ is given by $Z_{\Delta,r}(f)+C_{\Delta,r}(f)$,
where $C_{\Delta,r}(f)$ is a divisor supported at the cusps.
%
iii) The assumption that the coefficients $c(m,h)$ with $m=0$ be
rational is needed, since the space of holomorphic modular forms
$M_{1/2,L}$ may be non-trivial. Then the integrality of the $c(m,h)$
for $m<0$ does not automatically imply the rationality of all
coefficients.
\end{comment}
\end{remark}

\begin{corollary}
\label{cor:tbp} Let $F$ be a number field. If  $f\in M^!_{1/2,\tilde
\rho_L}$ has the property that all of its coefficients lie in $F$,
then the divisor $y_{\Delta,r}(f)$ vanishes in the Jacobian
$J(\Q(\sqrt{D}))\otimes_{\Z} F$.
\end{corollary}

\begin{proof}
The assertion follows from Theorems \ref{product} and
\ref{equivcond-borcherds}, using
%Remark \ref{rationalcoefficients} and
the Galois action on $M^!_{1/2,\tilde\rho_L}$.
\end{proof}

As another corollary, we obtain the following generalization of the
Gross-Kohnen-Zagier theorem \cite{GKZ}. It can derived from
Corollary \ref{cor:tbp} using Serre duality as in \cite{Bo2}.

\begin{corollary}
\label{GKZ-theorem} Let $\rho=\rho_L$ when $\Delta>0$, and let
$\rho=\bar\rho_L$ when $\Delta<0$. The generating series
\[
A_{\Delta,r}(\tau) =
%\leg{\Delta}{0}\calM_2^\vee\frake_0+
\sum_{h\in L'/L} \sum_{n>0
%m\in \Z-\sgn(\Delta)Q(h)
} y_{\Delta,r }(-n,h)q^n\frake_{h}
\]
is a cusp form of weight $3/2$ for $\tilde \Gamma$ of type
$\bar\rho$ with values in $J(\Q(\sqrt{\Delta}))$ {\text {\rm (}}i.e.
$A_{\Delta,r}\in S_{3/2,\bar\rho}\otimes_\Z
J(\Q(\sqrt{\Delta}))${\text {\rm )}}.
%
%It transforms with the
%representation $\bar\rho_L$ if $\Delta>0$ and with $\rho_L$ if
%$\Delta<0$.
%Here $\calM_2^\vee$ denotes the class of the dual of the
%line bundle of modular forms of weight $2$ on the modular curve
%$X_0(N)$.
\hfill $\square$
\end{corollary}

\section{Hecke eigenforms and isotypical components of the Jacobian}
\label{sect:jac}

Now we consider the implications of the results of the previous
section when the action of the Hecke algebra is introduced. We begin
with some notation. Let $L$ be the lattice of discriminant $2N$
defined in Section \ref{subsect:1.3}. Let $k\in
\frac{1}{2}\Z\setminus \Z$. The space of vector valued holomorphic
modular forms $M_{k,\rho_L}$ is isomorphic to the space of skew
holomorphic Jacobi forms $J_{k+1/2,N}^{skew}$ of weight $k+1/2$ and
index $N$. Moreover, $M_{k,\bar\rho_L}$ is isomorphic to the space
of holomorphic Jacobi forms $J_{k+1/2,N}$. There is an extensive
Hecke theory for Jacobi forms (see \cite{EZ}, \cite{Sk1},
\cite{SZ}), which gives rise to a Hecke theory on
 $M_{k,\rho_L}$ and  $M_{k,\bar\rho_L}$, and which is compatible with the Hecke theory on vector valued modular forms considered in \cite{BrSt}. In particular, there is an Atkin-Lehner theory for these spaces.

The subspace $S_{k,\rho_L}^{new}$ of newforms of $S_{k,\rho_L}$ is
isomorphic as a module over the Hecke algebra to the space of
newforms $S^{new,+}_{2k-1}(N)$ of weight $2k-1$ for $\Gamma_0(N)$ on
which the Fricke involution acts by multiplication with
$(-1)^{k-1/2}$. The isomorphism is given by the Shimura
correspondence. Similarly, the subspace $S_{k,\bar\rho_L}^{new}$ of
newforms of $S_{k,\bar\rho_L}$ is isomorphic as a module over the
Hecke algebra to the space of newforms $S^{new,-}_{2k-1}(N)$ of
weight $2k-1$ for $\Gamma_0(N)$ on which the Fricke involution acts
by multiplication with $(-1)^{k+1/2}$ (see \cite{SZ}, \cite{GKZ},
\cite{Sk1}). Observe that the Hecke $L$-series of any $G\in
S^{new,\pm}_{2k-1}(N)$ satisfies a functional equation under
$s\mapsto 2k-1-s$ with root number $\pm 1$. If $G\in
S^{new,\pm}_{2k-1}(N)$ is a normalized newform (in particular a
common eigenform of all Hecke operators), we denote by $F_G$ the
number field generated by the Hecke eigenvalues of $G$. It is well
known that we may normalize the preimage of $G$ under the Shimura
correspondence such that all its Fourier coefficients are contained
in $F_G$.

Let $\rho$ be one of the representations $\rho_L$ or $\bar\rho_L$.
For every positive integer $l$ there is a Hecke operator $T(l)$ on
$M_{k,\rho}$ which is self adjoint with respect to the Petersson
scalar product. The action on the Fourier expansion
$g(\tau)=\sum_{h,n} b(n,h)e(m\tau)\frake_h$ of any $g\in M_{k,\rho}$
can be described explicitly (for example, see \S4 of \cite{EZ}, or
\S0 (5) of \cite{SZ}). For instance, if $p$ is a prime coprime to
$N$ and we write $g\mid_{k} T(p)= \sum_{h,n}
b^*(n,h)e(n\tau)\frake_h$, we have
\begin{align}
\label{heckeop} b^*(n,h)=b(p^2n,ph)+p^{k-3/2}\leg{4N\sigma n}{p}
b(n,h)+p^{2k-2}b(n/p^2,h/p),
\end{align}
where $\sigma=1$ if $\rho=\rho_L$, and $\sigma=-1$ if
$\rho=\bar\rho_L$. There are similar formulas for general $l$.

The Hecke operators act on harmonic weak Maass forms and on weakly
holomorphic modular forms in an analogous way. In particular, the
formula for the action on Fourier coefficients is the same. In the
following we assume that $k\leq 1/2$. Recall from Section
\ref{sect:modularforms} that there is a bilinear pairing
$\{\cdot,\cdot\}$ between the spaces $S_{2-k,\bar\rho}$ and
$H_{k,\rho}$, which induces a non-degenerate pairing of
$S_{2-k,\bar\rho}$ and $H_{k,\rho}/M^!_{k,\rho}$.

\begin{proposition}
The Hecke operator $T(l)$ is (up to a scalar factor) self adjoint
with respect to the pairing $\{\cdot,\cdot\}$. More precisely we
have
\[
\{g,f\mid_{k} T(l)\} = l^{2k-2} \{g\mid_{2-k} T(l),f\}
\]
for any $g\in S_{2-k,\bar\rho}$ and $f\in H_{k,\rho}$.
\end{proposition}

\begin{proof}
From the definition of the Hecke operator or from its action on the
Fourier expansion of $f$ one sees that
\begin{align}
\label{tlxi} \xi_k (f\mid_k T(l)) = l^{2k-2} (\xi_k f) \mid_{2-k}
T(l).
\end{align}
Hence the assertion follows from the the self-adjointness of $T(l)$
with respect to the Petersson
scalar product on  $S_{2-k,\bar\rho}$.
\begin{comment}
Using the self-adjointness of $T(l)$ with respect to the Petersson
scalar product, we have
\begin{align*}
\{g,f\mid_{k} T(l)\}&= \big( g,\, \xi_k(f\mid T(l))\big)\\
&= l^{2k-2} \big( g, \,(\xi_k f)\mid  T(l) \big)\\
&= l^{2k-2} \big( g\mid  T(l),\, \xi_k f \big)\\
&= l^{2k-2} \{g\mid  T(l),\,f\} .
\end{align*}
This proves the proposition.
\end{comment}
\end{proof}

Recall that (see \cite{McG}) the space $S_{2-k,\bar\rho}$ has a
basis of cusp forms with rational coefficients. Let $F$ be any
subfield of $\C$. We denote by $S_{2-k,\bar\rho}(F)$ the $F$-vector
space of cusp forms in $S_{2-k,\bar\rho}$ with Fourier coefficients
in $F$. Moreover, we write $H_{k,\rho}(F)$ for the $F$-vector space
of harmonic weak Maass forms whose principal part has coefficients
in $F$. We write $M^!_{k,\rho}(F)$ for the subspace of weakly
holomorphic modular forms whose principal part has coefficients in
$F$. Using the pairing $\{\cdot,\cdot\}$, we identify
$H_{k,\rho}(F)/M^!_{k,\rho}(F)$ with the $F$-dual of
$S_{2-k,\bar\rho}(F)$.

\begin{lemma}
\label{pair1} Let $g\in S_{2-k,\bar\rho}$, and suppose that $f\in
H_{k,\rho}$ has the property that $\{g,f\}=1$, and also satisfies
$\{g',f\}=0$ for all $g'\in S_{2-k,\bar\rho}$ orthogonal to $g$.
Then $\xi_k(f)=\| g\|^{-2} g$, where $\| g\|$ denotes the Petersson
norm of $g$.\end{lemma}

\begin{proof}
This follows directly from the definition of the pairing.
\end{proof}

\begin{lemma}
\label{goodf} Let $F$ be a subfield of $\C$, and let $g\in
S_{2-k,\bar\rho}(F)$ be a newform. There is a $f \in H_{k,\rho}(F)$
such that
\[
\xi_k(f)=\| g\|^{-2} g.
\]
\end{lemma}

\begin{proof}
Since $g\in S_{2-k,\bar\rho}(F)$ is a newform, the orthogonal
complement of $g$ with respect to the Petersson scalar product has a
basis consisting of cusp forms with coefficients in $F$. Let
$g_2,\dots,g_d\in S_{2-k,\bar\rho}(F)$ be a basis of the orthogonal
complement of $g$. Let $f_1,\dots,f_d\in H_{k,\rho}(F)$ be the dual
basis of the basis $g,g_2,\dots,g_d$ with respect to $\{\cdot,\cdot
\}$. In particular $\{g,f_1\}=1$, and $\{g,f_j\}=0$ for all
$j=2,\dots,d$. According to Lemma \ref{pair1} we have that
$\xi_k(f_1)=\| g\|^{-2} g$. This completes the proof of the lemma.
\end{proof}

\begin{lemma}
\label{lem:hecke} Let $f\in H_{k,\rho}(F)$ and assume that
$\xi_k(f)\mid_{2-k} T(l) = \lambda_l \xi_k(f)$ with $\lambda_l\in
F$. Then
\[
f\mid_{k} T(l) - l^{2k-2} \lambda_l f\in M^!_{k,\rho}(F).
\]
\end{lemma}

\begin{proof}
The formula for the action of $T(l)$ on the Fourier expansion
implies that $f\mid_{k} T(l)\in H_{k,\rho}(F)$. Moreover, it follows
from \eqref{tlxi} that
\[
\xi_k\left( f\mid_{k} T(l) - l^{2k-2} \lambda_l f\right) =0.
\]
This proves the lemma.
\end{proof}

We now come to the main results of this section. Let $G\in
S_{2}^{new}(N)$ be a normalized newform of weight $2$ and write
$F_G$ for the number field generated by the eigenvalues of $G$. If
$G\in S_{2}^{new,-}(N)$, we put $\rho=\rho_L$ and assume that
$\Delta$ is a positive fundamental discriminant. If $G\in
S_{2}^{new,+}(N)$, we put $\rho=\bar \rho_L$ and assume that
$\Delta$ is a negative fundamental discriminant. There is a newform
$g\in S_{3/2,\bar\rho}^{new}$ mapping to $G$ under the Shimura
correspondence. We may normalize $g$ such that all its coefficients
are contained in $F_G$. Therefore by Lemma \ref{goodf}, there is a
harmonic weak Maass form $f\in H_{1/2, \rho}(F_G)$ such that
\[
\xi_{1/2}(f)=\|g\|^{-2} g.
\]
This form is  unique up to addition of a weakly holomorphic form in
$M^!_{1/2, \rho}(F_G)$.

\begin{theorem}
\label{goodpoint} The divisor $y_{\Delta,r}(f)\in
\Div(X_0(N))\otimes F_G$ defines a point in the $G$-isotypical
component of the Jacobian $J(\Q(\sqrt{\Delta}))\otimes\C$.
\end{theorem}

\begin{proof}
We write $\lambda_l$ for eigenvalue of the Hecke operator $T(l)$
corresponding to $G$ (where $l\in \Z_{>0}$). Let $p$ be any prime
coprime to $N$. It suffices to show that under the action of the
Hecke algebra on the Jacobian we have
\[
T(p) y_{\Delta,r}(f) = \lambda_p y_{\Delta,r}(f).
\]
%for all primes $p$ coprime to $N$. Hence let $p$ be a prime coprime to $N$.
It is easily seen that
\begin{align}
\label{heckejac} T(p) y_{\Delta,r}(m,h) = y_{\Delta,r}(p^2 m,p h) +
\leg{4N\sigma m}{p} y_{\Delta,r}(m,h) + p y_{\Delta,r}(m/p^2,h/p) ,
\end{align}
where $m\in \Q$ is  negative, $h\in L'/L$,  and
$\sigma=\sgn(\Delta)$. (For example, see p. 507 and p. 542 of
\cite{GKZ} for the case $\Delta=1$. The general case is analogous.)
%\texttt{Check!!!}
Combining this with \eqref{heckeop}, we see that
\begin{align}
\label{h1} T(p) y_{\Delta,r}(f) = p y_{\Delta,r}(f\mid_{1/2} T(p)).
\end{align}
In view of Lemma \ref{lem:hecke}, there is a $f'\in M^!_{1/2,
\rho}(F_G)$ such that
\[
f\mid_{1/2} T(p) = p^{-1} \lambda_p f+ f'.
\]
Combining this with \eqref{h1}, we find that
\begin{align*}
T(p) y_{\Delta,r}(f) = \lambda_p y_{\Delta,r}(f)+ p y_{\Delta,r}(f').
\end{align*}
But Lemma \ref{rationalcoeff} and Corollary \ref{cor:tbp} imply that
$y_{\Delta,r}(f')$ vanishes in $J(\Q(\sqrt{\Delta}))\otimes\C$.
\end{proof}

\begin{theorem}
\label{thm:alg1} Assume that $\Delta\neq 1$. Let $f\in H_{1/2,
\rho}(F_G)$ be a weak Maass form such that $\xi_{1/2}(f)$ is a
newform which maps to $G\in S_2^{new}(N)$ under the Shimura
correspondence. Denote the Fourier coefficients of $f^+$ by
$c^+(m,h)$. Then the following are equivalent:
\begin{enumerate}
\item[(i)]
The Heegner divisor $y_{\Delta,r}(f)$ vanishes in
$J(\Q(\sqrt{\Delta}))\otimes\C$.
\item[(ii)]
The coefficient $c^+(\tfrac{|\Delta|}{4N},\tfrac{r}{2N})$ is
algebraic.
\item[(iii)]
The coefficient $c^+(\tfrac{|\Delta|}{4N},\tfrac{r}{2N})$ is
contained in $F_G$.
\end{enumerate}
\end{theorem}

\begin{proof}
If (i) holds,
%Then $y_{\Delta,r}(f^\sigma)$ vanishes in $J(\Q(\sqrt{\Delta}))\otimes\C$ for all $\sigma\in \Gal(F_G/\Q)$..
%Then, according to  Theorem \ref{scholl}, the canonical differential of the third kind $\eta_{\Delta,r}(f)$ corresponding to $y_{\Delta,r}(f)$ is defined over $\bar \Q$.
then Theorem \ref{equivcond} implies that the coefficients
$c^+(\tfrac{|\Delta|n^2}{4N},\tfrac{rn}{2N})$ are in $F_G$ for all
positive integers $n$. Hence (iii) holds. Moreover, it is clear that
(iii) implies (ii).

Now we show that (ii) implies (i). If $y_{\Delta,r}(f)\neq 0$ is in
$J(\Q(\sqrt{\Delta}))\otimes\C$, then Theorem \ref{equivcond}
implies that there are positive integers $n$ for which
$c^+(\tfrac{|\Delta|n^2}{4N},\tfrac{rn}{2N})$ is transcendental. Let
$n_0$ be the smallest of these integers.
% for which $c^+(\tfrac{|\Delta|n_0^2}{4N},\tfrac{rn_0}{2N})$ is transcendental.
We need to show that $n_0=1$.

Assume that $n_0\neq 1$, and that $p\mid n_0$ is prime.
% and write $n_0=p^2 n_1$.
%Then   $c^+(\tfrac{|\Delta|n_1^2}{4N},\tfrac{rn_1}{2N})$ is algebraic.
%
Let $\lambda_p$ be the eigenvalue of the Hecke operator $T(p)$
corresponding to $G$. By Lemma \ref{lem:hecke}, there is a $f'\in
M^!_{1/2, \rho}(F_G)$ such that
\[
f\mid_{1/2} T(p) = p^{-1} \lambda_p f+ f'.
\]
Using the formula for the action of $T(p)$ on the Fourier expansion
of $f$, we see that
$c^+(\tfrac{|\Delta|n_0^2}{4N},\tfrac{rn_0}{2N})$ is an algebraic
linear combination of Fourier coefficients
$c^+(\tfrac{|\Delta|n_1^2}{4N},\tfrac{r_1 n_1}{2N})$ of $f$ with
$n_1\leq n_0/p$ and coefficients of $f'$. In view of Lemma
\ref{rationalcoeff} and Lemma \ref{weakrat}, this implies that
$c^+(\tfrac{|\Delta|n_0^2}{4N},\tfrac{rn_0}{2N})$ is algebraic,
contradicting our assumption.
\end{proof}

\begin{remark}
\label{orthcusp3} Theorem \ref{thm:alg1} also holds for $\Delta=1$
when $S_{1/2,\rho}=0$. More generally, it should also hold for
$\Delta=1$ when $f$ is chosen to be orthogonal to the cusp forms in
$S_{1/2,\rho}$. The latter condition ensures that
Theorem~\ref{equivcond} still applies, see Remark \ref{orthcusp1}.
However, it remains to show that a weakly holomorphic form $f'\in
M^!_{1/2, \rho}(F_G)$ which is orthogonal to cusp forms, automatically
has {\em all} coefficients in $F_G$.
\end{remark}

Using the action of the Hecke algebra on the Jacobian, we may derive
a more precise version of Corollary \ref{GKZ-theorem}. Let
$y^G_{\Delta,r}(m,h)$ denote the projection of the Heegner divisor
$y_{\Delta,r}(m,h)$ onto its $G$-isotypical component. We consider
the generating series
\[
A^G_{\Delta,r}(\tau)= \sum_{h\in L'/L} \sum_{n>0
%m\in \Z-\sgn(\Delta)Q(h)
} y^G_{\Delta,r }(-n,h)q^n\frake_{h}\in S_{3/2,\bar\rho}\otimes_\Z
J(\Q(\sqrt{\Delta})).
\]

\begin{theorem}
\label{isogen} We have the identity
\[
A^G_{\Delta,r}(\tau)= g(\tau)\otimes y_{\Delta,r}(f).
\]
In particular, the space in $(J(\Q(\sqrt{\Delta}))\otimes \C)^G$
spanned by the $y^G_{\Delta,r}(m,h)$ is at most one-dimensional and
is generated by $y_{\Delta,r}(f)$.
\end{theorem}

\begin{proof}
We write $\lambda_l$ for eigenvalue of the Hecke operator $T(l)$
corresponding to $G$ (where $l\in \Z_{>0}$). Let $p$ be any prime
coprime to $N$. By means of \eqref{heckeop} and \eqref{heckejac}, we
see that
\[
T(p) A_{\Delta,r} = A_{\Delta,r} \mid_{3/2} T(p),
\]
where on the left hand side the Hecke operator acts through the
Jacobian, while on the right hand side it acts through
$S_{3/2,\bar\rho}$. Consequently, for the $G$-isotypical part we
find
\[
A^G_{\Delta,r} \mid_{3/2} T(p) = \lambda_p A^G_{\Delta,r}.
\]
Hence $A^G_{\Delta,r}$ is an eigenform of all the $T(p)$ for $p$
coprime to $N$ with the same eigenvalues as $g$. By ``multiplicity
one'' for  $S_{3/2,\bar\rho}^{new}$, we find that $A^G_{\Delta,r}= C
g$ for some constant  $C\in (J(\Q(\sqrt{\Delta}))\otimes \C)^G$. To
compute the constant, we determine the pairing with $f$. We have
\begin{align*}
\{A^G_{\Delta,r},f\}&=y_{\Delta,r}(f),\\
\{C g,f\}&=C (g,\xi_{1/2}(f))= C.
\end{align*}
This concludes the proof of the theorem.
\end{proof}

Note that in the case $\Delta=1$, the theorem (in a different
formulation) was proved in a different way in \cite{GKZ}. We may use
the Gross-Zagier theorem to relate the vanishing of
$y_{\Delta,r}(f)$ to the vanishing of a twisted $L$-series
associated with $G$.

\begin{theorem}
\label{thm:alg2} Let the hypotheses be as in Theorem \ref{thm:alg1}.
The following are equivalent.
\begin{enumerate}
\item[(i)]
The Heegner divisor $y_{\Delta,r}(f)$ vanishes in
$J(\Q(\sqrt{\Delta}))\otimes\C$.
\item[(ii)]
We have $L'(G,\chi_\Delta,1)=0$.
\end{enumerate}
\end{theorem}

\begin{proof}
We denote the Fourier coefficients of $g$ by $b(n,h)$ for $n\in
\Z-\sgn(\Delta)Q(h)$ and $h\in L'/L$. Since $g$ is a newform, by
Lemma 3.2 of \cite{SZ}, there is a fundamental discriminant $d$
coprime to $N$ such that $\sgn(\Delta)d<0$ and such that $b(n,h)\neq
0$ for $n=-\sgn(\Delta)\frac{d}{4N}$ and some $h\in L'/L$. According
to Corollary 1 of Chapter II in \cite{GKZ}, and \cite{Sk2}, we have
the Waldspurger type formula
\[
|b(n,h)|^2=\frac{\|g\|^2}{2\pi \|G\|^2}\sqrt{|d|}L(G,\chi_d,1).
\]
In particular, the non-vanishing of $b(n,h)$ implies the
non-vanishing of  $L(G,\chi_d,1)$.

On the other hand, it follows from the Gross-Zagier formula (Theorem
6.3 in \cite{GZ}) that the global Neron-Tate height on $J(H)$ of
$y_{\Delta,r}^G(-n,h)$ is given by
\[
\langle y_{\Delta,r}^G(-n,h),y_{\Delta,r}^G(-n,h)\rangle = \frac{
h_K u^2}{8\pi^2 \|G\|^2}\sqrt{|d\Delta|}
L'(G,\chi_\Delta,1)L(G,\chi_d,1).
\]
Here $H$ is the Hilbert class field of $K=\Q(\sqrt{d\Delta})$, and
$2u$ is the number of roots of unity in $K$, and $h_K$ denotes the
class number of $K$.

Consequently, the Heegner divisor $y_{\Delta,r}^G(-n,h)$ vanishes if
and only if $L'(G,\chi_\Delta,1)$ vanishes. But by Theorem
\ref{isogen} we know that
\[
y_{\Delta,r}^G(-n,h)= y_{\Delta,r}(f) b(n,h).
\]
This concludes the proof of the theorem.
\end{proof}

As described in the introduction, the results in this section imply
Theorem~\ref{Lvalues}. We conclude this section with the proof of
Corollary~\ref{estimates}.

\begin{proof}[Proof of Corollary~\ref{estimates}]
By Theorem~\ref{Lvalues}, it suffices to show that
\begin{equation}\label{estimate1}
\# \{ 0\leq \Delta < X \ \ {\text {fundamental}}\ : \
L(G,\chi_{\Delta},1)\neq 0\}\gg_G \frac{X}{\log X},
\end{equation}
and
\begin{equation}\label{estimate2}
\# \{ -X< \Delta < 0 \ \ {\text {fundamental}}\ : \
L'(G,\chi_{\Delta},1)\neq 0\}\gg_{G,\epsilon} X^{1-\epsilon}.
\end{equation}
Corollary 3 of \cite{OSk} implies (\ref{estimate1}), and the proof
of Theorem 1 of \cite{PP} implies (\ref{estimate2}).
\end{proof}

\section{Examples}
\label{sect:ex}

Here we give some examples related to the main results in this
paper.

\subsection{Twisted modular polynomials}

Here we use Theorems \ref{product} and \ref{equivcond-borcherds} to
deduce the infinite product expansion of twisted modular polynomials
found by Zagier (see \S 7 of \cite{Za2}).

Let $N=1$. Then we have $L'/L\cong \Z/2\Z$. Moreover,
$H_{1/2,\rho_L}=M^!_{1/2,\rho_L}$ and $H_{1/2,\bar\rho_L}=0$.
Therefore we consider the case where $\Delta>1$ is a positive
fundamental discriminant. Let $r\in \Z$ such that $\Delta\equiv
r^2\pmod{4}$. By \S 5 of \cite{EZ}, the space $M^!_{1/2,\rho_L}$ can
be identified with the space $M^!_{1/2}$ of scalar valued weakly
holomorphic modular forms of weight $1/2$ for $\Gamma_0(4)$
satisfying the Kohnen plus space condition. For every negative
discriminant $d$, there is a unique $f_{d}\in M^!_{1/2}$, whose
Fourier expansion at the cusp $\infty$ has the form
\[
f_{d}=q^{d}+ \sum_{\substack{n\geq 1\\n\equiv 0,1\;(4)}} c_d(n)q^n.
\]
The expansions of the first few $f_{d}$ are given in \cite{Za2}, and
one sees that the coefficients are rational. Theorems \ref{product}
and \ref{equivcond-borcherds} gives a meromorphic modular form
$\Psi_\Delta(z,f_d):=\Psi_{\Delta,r}(z,f_d)$ of weight $0$ for the
group $\Gamma=\Sl_2(\Z)$ whose divisor on $X(1)$ is given by
\[
Z_{\Delta}(d):=Z_{\Delta,r}(d/4,d/2)=\sum_{\lambda \in L_{\Delta
d}/\Gamma} \frac{\chi_\Delta(\lambda)}{w(\lambda)}\cdot  Z(\lambda).
\]
By \eqref{eq:qf},  $L_{\Delta d}/\Gamma$ corresponds to the
$\Gamma$-classes of integral binary quadratic forms of discriminant
$\Delta d$. Moreover, for sufficiently large $\Im(z)$, we have the
product expansion
\begin{align}
\label{prodn1} \Psi_\Delta(z,f_d)=\prod_{n=1}^{\infty}
\prod_{b\;(\Delta)} [1-e(nz+b/\Delta)]^{\leg{\Delta}{b}c_d(\Delta
n^2)}.
\end{align}
From these properties it follows that
\begin{align}
\Psi_\Delta(z,f_d)=\prod_{\lambda \in L_{\Delta d}/\Gamma} \big(
j(z)-j(Z(\lambda))\big)^{\chi_\Delta(\lambda)}.
\end{align}

As an example, let $\Delta:=5$ and $d:=-3$. There are two classes of
binary quadratic forms of discriminant $-15$, represented by
$[1,1,4]$ and $[2,1,2]$, and their corresponding CM points are
$\frac{-1+\sqrt{-15}}{2}$ and $\frac{-1+\sqrt{-15}}{4}$. It is well
known that the singular moduli of $j(\tau)$ of these points are
$-\frac{191025}{2}-\frac{85995}{2}\sqrt{5}$, and
$-\frac{191025}{2}+\frac{85995}{2}\sqrt{5}$. The function $f_{-3}$
has the Fourier expansion
\[
f_{-3}= q^{-3}-248 \,q+ 26752\, q^4 - 85995\, q^5 + 1707264\, q^8 -
4096248\, q^9 + \dots .
\]
Multiplying out the product over $b$ in \eqref{prodn1}, we obtain
the infinite product expansion
\[
\Psi_5(z,f_{-3})=
\frac{j(z)+\frac{191025}{2}+\frac{85995}{2}\sqrt{5}}{j(z)+\frac{191025}{2}-\frac{85995}{2}\sqrt{5}}
= \prod_{n=1}^{\infty}
\left(\frac{1+\frac{1-\sqrt{5}}{2}q^n+q^{2n}}{1+\frac{1+\sqrt{5}}{2}q^n+q^{2n}}\right)^{c_{-3}(5
n^2)}.
\]

\subsection{Ramanujan's mock theta functions $f(q)$ and $\omega(q)$}
\label{sect:8.2}

Here we give an example of a Borcherds product arising from
Ramanujan's mock theta functions.
We first recall the modular transformation
properties of $f(q)$, defined in (\ref{fq}), and
\begin{equation}\label{omega}
\begin{split}
\omega(q):=\sum_{n=0}^{\infty} \frac{q^{2n^2+2n}}{(q;q^2)_{n+1}^2} =
\frac{1}{(1-q)^2}+\frac{q^4}{(1-q)^2(1-q^3)^2}+\frac{q^{12}}{
(1-q)^2(1-q^3)^2(1-q^5)^2}+\dots.
\end{split}
\end{equation}
Using these functions, define the vector valued function $F(\tau)$
by
\begin{equation}\label{F}
F(\tau)=(F_0(\tau),\ F_1(\tau),\
F_2(\tau))^{T}:=(q^{-\frac{1}{24}}f(q), \
2q^{\frac{1}{3}}\omega(q^{\frac{1}{2}}),\
2q^{\frac{1}{3}}\omega(-q^{\frac{1}{2}}))^{T}.
\end{equation}
Similarly, let $G(\tau)$ be the vector valued non-holomorphic
function defined by
\begin{equation}\label{G}
G(\tau)=(G_0(\tau),\ G_1(\tau),\
G_2(\tau))^{T}:=2i\sqrt{3}\int_{-\overline{\tau}}^{i\infty}
\frac{(g_1(z),\ g_0(z),\ -g_2(z))^{T}}{ \sqrt{-i(\tau+z)}} \ dz,
\end{equation}
where the $g_i(\tau)$ are the cuspidal weight $3/2$ theta functions
\begin{equation}\label{gi}
\begin{split}
g_0(\tau)&:=\sum_{n=-\infty}^{\infty}
(-1)^n\left(n+\frac{1}{3}\right)e^{3\pi i \left(n+\frac{1}{3}
\right)^2 \tau},\\
g_1(\tau)&:=-\sum_{n=-\infty}^{\infty}
\left(n+\frac{1}{6}\right)e^{3\pi i
\left (n+\frac{1}{6}\right )^2 \tau},\\
g_2(\tau)&:=\sum_{n=-\infty}^{\infty} \left(n+\frac{1}{3}\right
)e^{3\pi i \left(n+\frac{1}{3} \right)^2 \tau}.\end{split}
\end{equation}
Using these vector valued functions, Zwegers \cite{Z1} let
$H(\tau):=F(\tau)-G(\tau)$,
and he proved \cite{Z1} that it is a vector valued weight $1/2$
harmonic weak Maass form. In particular, it satisfies
\begin{align}
\label{smallrep1}
H(\tau+1)&=\left(\begin{matrix}\zeta_{24}^{-1} & 0 & 0\\
0 & 0 & \zeta_3\\
0 &\zeta_3& 0\end{matrix}
\right) H(\tau),\\
\label{smallrep2} H(-1/\tau)&= \sqrt{-i\tau} \cdot \left(
\begin{matrix}
0&1&0\\
1&0&0\\
0&0&-1\end{matrix} \right) H(\tau).
\end{align}

Now let $N:=6$. One can check the following lemma which asserts that this
representation of $\tilde\Gamma$ is an irreducible piece of the Weil
representation $\bar\rho_L$.

\begin{lemma}
Assume that $H=(h_0,h_1,h_2)^T:\H\to \C^3$ is a vector valued
modular form of weight $k$ for $\tilde\Gamma$ transforming with the
representation defined by \eqref{smallrep1} and \eqref{smallrep2}.
Then the function
\begin{align}
\label{bigrep1} \tilde H=(0,h_0,h_2-h_1,0,-h_1-h_2, -h_0, 0, h_0,
h_1+h_2, 0, h_1-h_2, -h_0)^T
\end{align}
is a vector valued modular form of weight $k$ for $\tilde\Gamma$
with representation $\bar\rho_L$. Here we have identified $\C[L'/L]$
with $\C^{12}$ by mapping the standard basis vector of $\C[L'/L]$
corresponding to the coset $j/12+\Z\in L'/L$  to the standard basis
vector $e_j$ of $\C^{12}$ (where $j=0,\dots,11$).
\end{lemma}

This lemma shows that $H$ gives rise to an element $\tilde H\in
H_{1/2,\bar\rho_L}$. Let $c^\pm(m,h)$ be the coefficients of
$\tilde{H}$.  For any fundamental discriminant $\Delta<0$ and any
integer $r$ such that $\Delta\equiv r^2\pmod{24}$, we obtain a
twisted generalized Borcherds lift $\Psi_{\Delta,r}(z,\tilde{H})$.
By Theorems~\ref{product} and \ref{equivcond-borcherds}, it is a
weight 0 meromorphic modular function on $X_0(6)$ with divisor
\[
%\dv(\Psi_{\Delta,r}(z,\tilde{H}))=
2Z_{\Delta,r}(-\tfrac{1}{24},\tfrac{1}{12})
-2Z_{\Delta,r}(-\tfrac{1}{24},\tfrac{5}{12}).
\]
Moreover, it has the infinite product expansion
\begin{align}
\label{ex2prod} \Psi_{\Delta,r }(z,\tilde H)=
\prod_{\substack{n=1}}^\infty
P_\Delta\left(q^n\right)^{c^+(|\Delta|n^2/24,\,rn/12)},
\end{align}
where
\begin{align}
P_\Delta(X):=\prod_{b\;(\Delta)}\left[1-e(b/\Delta)X
\right]^{\leg{\Delta}{b}}.
\end{align}

For instance, let $\Delta:=-8$ and $r:=4$. The set
$L_{-8,4}/\Gamma_0(6)$ is represented by the binary quadratic forms
$Q_1=[6,4,1]$ and $Q_2=[-6,4,-1]$, and $L_{-8,-4}/\Gamma_0(6)$ is
represented by $-Q_1$ and $-Q_2$. The Heegner points in $\H$
corresponding to $Q_1$ and $Q_2$ respectively are
\begin{align*}
\alpha_1&=\frac{-2+\sqrt{-2}}{6},& \alpha_2&=\frac{2+\sqrt{-2}}{6}.
\end{align*}
Consequently, the divisor of $\Psi_{-8,4}(z,\tilde{H})$ on $X_0(6)$
is given by $2 (\alpha_1)-2(\alpha_2)$. In this case the infinite
product expansion \eqref{ex2prod} only involves the coefficients of
the components of $\tilde H$ of the form $\pm (h_1+h_2)$. To
simplify the notation, we put
\begin{align*}
-2q^{1/3}\big( \omega(q^{1/2})+ \omega(-q^{1/2})\big)
=:\sum_{n\in\Z+1/3} a(n)q^n
=-4\,q^{1/3}-12\,q^{4/3}-24\,q^{7/3}-40\,q^{10/3}-\dots.
%72\,q^{13/3}+\dots.
%-116\,q^{16/3}
%-176\,q^{19/3}
%-272*q^(22/3)-404*q^(25/3)-584*q^(28/3)+O(q^(31/3))
\end{align*}
We have
\[
P_{-8}(X):=\frac{1+\sqrt{-2}X-X^2}{1-\sqrt{-2}X-X^2},
\]
and the infinite product expansion \eqref{ex2prod} can be rewritten
as
\begin{align}
\label{ex2prod2} \Psi_{-8,4 }(z,\tilde H)=
\prod_{\substack{n=1}}^\infty P_{-8}\left(q^n\right)^{\leg{n}{3}
a(n^2/3)}.
\end{align}

It is amusing to work out an expression for $\Psi_{-8,4
}(z,\tilde H)$. We use the Hauptmodul for $\Gamma^*_0(6)$, the
extension of $\Gamma_0(6)$ by all Atkin-Lehner involutions, which is
\begin{align*}
j_6^*(z)=\left( \frac{\eta(z)\eta(2z)}{\eta (3 z)\eta(6z)}\right)^4
+ 4 +3^4 \left( \frac{\eta(3z)\eta(6z)}{\eta (
z)\eta(2z)}\right)^4
=q^{-1}+
79\,q+352\,q^2+1431\,q^3+\dots.
%4160\,q^4+\dots.
%13015\,q^5+31968\,q^6+\dots.
%+81162\,q^7+183680*q^8+O(q^9),q,9)
\end{align*}
Here $\eta(z)=q^{1/24} \prod_{n=1}^\infty(1-q^n)$ denotes the
Dedekind eta function. We have
$j_6^*(\alpha_1)=j_6^*(\alpha_2)=-10$. Hence $j_6^*(z)+10$ is a
rational function on $X_0(6)$ whose divisor consists of the $4$
cusps with multiplicity $-1$ and the points $\alpha_1$, $\alpha_2$
with multiplicity $2$. The unique normalized cusp form of weight 4
for $\Gamma^*_0(6)$ is
\begin{align*}
\delta(z):=\eta(z)^2\eta(2z)^2\eta (3 z)^2\eta(6z)^2
=q-2\,q^2-3\,q^3+4\,q^4+6\,q^5+6\,q^6-16\,q^7-8\,q^8+\dots.
%-12\,q^{10}
\end{align*}
Using these functions, we find that
\[
\phi(z):=\Psi_{-8,4 }(z,\tilde H)\cdot
(j^*_6(z)+10)\delta(z)
\]
is a holomorphic modular form of weight $4$ for $\Gamma_0(6)$ with
divisor $4(\alpha_1)$. Using the classical
Eisenstein series, it turns out that
\begin{align*}
450 \phi(z) &= (3360-1920 \sqrt{-2})\delta(z)
+(1-7\sqrt{-2})E_4(z) +(4-28\sqrt{-2})E_4(2z) \\
&\phantom{=}{}+(89+7\sqrt{-2})E_4(3z) +(356+28\sqrt{-2})E_4(6z)  .
\end{align*}
Putting this all together, (\ref{ex2prod2}) becomes
\begin{displaymath}\begin{split}
\prod_{\substack{n=1}}^\infty &\left(
\frac{1+\sqrt{-2}q^n-q^{2n}}{1-\sqrt{-2}q^n-q^{2n}}\right)^{\leg{n}{3}
a(n^2/3)}=\frac{\phi(z)}{(j^*_6(z)+10)\delta(z)}\\
&\ \ =1-8\sqrt{-2}q-(64-24\sqrt{-2})q^2+(384+168\sqrt{-2})q^3+(64-1768\sqrt{-2})q^4+\dots.
\end{split}
\end{displaymath}

\subsection{Relations among Heegner points and vanishing derivatives of $L$-functions}
\label{sect:8.3}

Let $G\in S_2^{new}(\Gamma_0(N))$ be a newform of weight $2$, and
let $\Delta$ be a fundamental discriminant such that
$L(G,\chi_\Delta,s)$ has an odd functional equation. By the
Gross-Zagier formula, the vanishing of $L'(G,\chi_\Delta,1)$ is
equivalent to the vanishing of a certain Heegner divisor in the
Jacobian. More precisely,  let $f$ be a harmonic weak Maass form of
weight $1/2$ corresponding to $G$ as in Section  \ref{sect:jac}.
Then $L'(G,\chi_\Delta,1)$ vanishes if and only if the divisor
$y_{\Delta,r}(f)$ vanishes. If $G$ is defined over $\Q$, we may
consider the generalized regularized theta lift of $f$ as in
Section~\ref{sect:bor}. It gives rise to a rational function
$\Psi_{\Delta,r}(z,f)$ on $X_0(N)$ with divisor $y_{\Delta,r}(f)$.
(If $G$ is not defined over $\Q$, one also has to consider the
Galois conjugates.) Such relations among Heegner divisors cannot be
obtained as the Borcherds lift of a weakly holomorphic form. They
are given by the generalized regularized theta lift of a harmonic
weak Maass form.

As an example, we consider the the relation for Heegner points of
discriminant $-139$ on $X_0(37)$ found by Gross (see \S 4 of
\cite{Za1}). Let $N:=37$, $\Delta:=-139$, and $r:=3$. In our
notation, $L_{-139,3}/\Gamma_0(37)$ can be represented by the
quadratic forms
\begin{align*}
Q_1&=[37,3,1],& Q_2&=[185,151,31], &  Q_3&=[185,-71,7],\\
Q_1'&=[-37,3,-1],& Q_2'&=[-185,151,-31],& Q_3'&=[-185,-71,-7].\\
\end{align*}
Denote the corresponding points on $X_0(37)$ by $\alpha_1$,
$\alpha_2$, $\alpha_3$ and $\alpha_1'$, $\alpha_2'$, $\alpha_3'$.
Hence we have
\begin{align*}
Z_{1,1}(-\tfrac{139}{4\cdot37},\tfrac{3}{2\cdot 37}) &=
\alpha_1+\alpha_2+\alpha_3+
\alpha_1'+\alpha_2'+\alpha_3',\\
Z_{-139,3}(-\tfrac{1}{4\cdot37},\tfrac{1}{2\cdot 37})&=
\alpha_1+\alpha_2+\alpha_3- \alpha_1'-\alpha_2'-\alpha_3'.
\end{align*}
Gross proved that the function
\[
r(z)=\frac{\eta(z)^2}{\eta(37 z)^2} -\frac{3+\sqrt{-139}}{2}
\]
on $X_0(37)$ has the divisor
$(\alpha_1)+(\alpha_2)+(\alpha_3)-3(\infty)$. This easily implies
that $r'(z)$, the image of $r(z)$ under complex conjugation,
%=\eta(z)^2\eta(37 z)^{-2} -\frac{3-\sqrt{-139}}{2}$
has the divisor $(\alpha_1')+(\alpha_2')+(\alpha_3')-3(\infty)$.

We show how the function $r(z)$ can be obtained as a regularized
theta lift. Let $f_{139}\in H_{1/2,\rho_L}$ be the unique harmonic weak
Maass form whose Fourier expansion is of the form
\[
f_{139}=e(-\tfrac{139}{4\cdot37} \tau)\frake_{3} +
e(-\tfrac{139}{4\cdot37}\tau)\frake_{-3} + O(e^{-\eps v}),\qquad
v\to \infty.
\]
It is known that the dual space  $S_{3/2,\bar\rho_L}$ is
one-dimensional. Moreover, any element has the property that the
coefficients with index $\tfrac{139}{4\cdot37}$ vanish (see
\cite{EZ}, p.145). It follows from \eqref{pairalt} that
$\xi_{1/2}(f_{139})=0$, and so $f_{139}$ is weakly holomorphic. Its
Borcherds lift is equal to
\begin{align}
\label{ex3lift1} \Psi_{1,1}(z,f_{139})=r(z)\cdot r'(z)\cdot\frac{\eta(37
z)^2}{\eta(z)^2}.
\end{align}

On the other hand, we consider the unique  harmonic weak Maass form
$f_{1}\in H_{1/2,\bar\rho_L}$ whose Fourier expansion is of the form
\[
f_{1}=e(-\tfrac{1}{4\cdot37} \tau)\frake_{1} -
e(-\tfrac{1}{4\cdot37}\tau)\frake_{-1} + O(e^{-\eps v}),\qquad v\to
\infty.
\]
It is known that the dual space  $S_{3/2,\rho_L}$ is
two-dimensional. For a fixed $\lambda_0$ in the positive definite one-dimensional sublattice $K\subset L$, the theta series 
\[
g_0(\tau)=\sum_{\lambda\in K'} (\lambda,\lambda_0)\cdot q^{Q(\lambda)}\frake_\lambda
 \]
% associated to the one dimensional positive definite lattice $K$ with the harmonic polynomial
%$p(x)=(x_0,x)$ (for $x\in K\otimes_\Z\R$ and a fixed $x_0\in K$) 
is a non-zero element. Under the Shimura correspondence it is mapped to the Eisenstein series in
$M_2^+(\Gamma_0(37))$. Let $g_1$ be a generator of the orthogonal complement of $g_0$ in $S_{3/2,\rho_L}$. Then $g_1$ is a Hecke eigenform and we may normalize it such that it has rational coefficients. The Shimura lift of $g_1$ is the newform
$G_1\in S_2^+(\Gamma_0(37))$, 
which corresponds to the
conductor $37$ elliptic curve
\[
E: \ \ y^2=x^3+10x^2-20x+8.
\]
Its $L$-function has an even
functional equation, and it is known that $L(G_1,1)\neq 0$. By the
Waldspurger type formula for skew holomorphic Jacobi forms, see
\cite{Sk2}, this implies that the coefficients of $g_1$ with index
$\tfrac{1}{4\cdot37}$ do not vanish.
In view of \eqref{pairalt}, we
find that
\[
\xi_{1/2}(f_{1})=c_0 g_0 +c_1 g_1
\]
with non-zero constants $c_0$ and $c_1$.
So $f_{1}$
is not weakly holomorphic. Nevertheless, we may look at the twisted
generalized Borcherds lift of $f_{1}$. We obtain that
\begin{align}
\label{ex3lift2} \Psi_{-139,3}(z,f_{1})=r(z)/
r'(z)=\frac{\eta(z)^2-\frac{3+\sqrt{-139}}{2}\eta(37z)^2}{\eta(z)^2-\frac{3-\sqrt{-139}}{2}\eta(37z)^2}.
\end{align}
F.~Str\"omberg computed a large number of coefficients of $f_{1}$ numerically.
The first few coefficients of the holomorphic part of $f_{1}$ (indexed by the corresponding discriminants) are listed in
Table \ref{table:2}. Details on the computations and some further results will be given in \cite{BrStr}.
\begin{table}[h]
\caption{\label{table:2} Coefficients of $f_{1}$}
\begin{tabular}{|r|cc| }
\hline \rule[-3mm]{0mm}{8mm}
$\Delta$ && $c^+(-\Delta)$\\
\hline% \rule[-3mm]{0mm}{8mm}
$-3$  && $-0.324428362769321517518\dots$ \\%26731$\\
$-4$  && $-0.259821199677656112490\dots$ \\%69658$\\
$-7$  && $-0.436656751226664126195\dots$ \\%04156$\\
$-11$ && $-0.137166725483836081720\dots$ \\%73346$\\
%$-40$ && $-0.577676447138672293721\dots$ \\%10457$\\
$\vdots$ && $\vdots$ \\
$-136$&&  $0.053577466885218739004\dots$ \\%50497$\\
$-139$&& $0$\\ %1.8634999382929359284261114E-30$\\
$-151$&& $-0.277118960597558973488\dots$ \\%74948$\\
$\vdots$ && $\vdots$ \\
$-815$&& $-0.351965626356359803714\dots$ \\%19969$\\
$-823$&& $-1$\\ %-1.0000000000000000000009981E0$\\
$-824$&& $-0.666202201365835224525\dots$ \\%28153$\\
\hline
\end{tabular}
\end{table}
The rationality of the coefficients $c^+(139)$ and $c^+(823)$ corresponds to the vanishing of the Heegner divisors $Z_{-139,3}(-\tfrac{1}{4\cdot37},\tfrac{1}{2\cdot 37})$ and $Z_{-823,19}(-\tfrac{1}{4\cdot37},\tfrac{1}{2\cdot 37})$ in the Jacobian of $X_0(37)$.

To obtain an element of $H_{1/2,\bar\rho_L}$ corresponding to $g_1$ as in Lemma \ref{goodf}
we consider the unique  harmonic weak Maass form
$f_{12}\in H_{1/2,\bar\rho_L}$ whose Fourier expansion is of the form
\[
f_{12}=e(-\tfrac{12}{4\cdot37} \tau)\frake_{30} -
e(-\tfrac{12}{4\cdot37}\tau)\frake_{-30} + O(e^{-\eps v}),\qquad v\to
\infty.
\]
Arguing as above we see that $\xi_{1/2}(f_{12})$ is a non-zero multiple of $g_1$.
Table \ref{table:3} includes some of the coefficients of $f_{12}$ and the corresponding values of 
$L'(G_1,\chi_\Delta,1)$ as numerically computed by F. Str\"omberg. We have that $L'(G_1,\chi_{-139},1)=L'(G_1,\chi_{-823},1)=0$ by the Gross-Zagier formula.
We see that the values are in accordance with 
Theorem~\ref{thm:alg1} and Theorem~\ref{thm:alg2}.

\begin{table}[h]
\caption{\label{table:3} Coefficients of $f_{12}$}
\begin{tabular}{|r|cc|cc| }
\hline \rule[-3mm]{0mm}{8mm}
$\Delta$ && $c^+(-\Delta)$&& $L'(G_1,\chi_\Delta,1)$ \\
\hline %\rule[-3mm]{0mm}{8mm}
$-3$   && $1.026714911692035447445\dots$ &&$1.47929949207700\dots$\\
$-4$   && $1.220536400967031662527\dots$ &&$1.81299789721820\dots$\\
$-7$   && $1.690029746320007621414\dots$ &&$2.11071898017914\dots$\\
$-11$  && $0.588499823548491754837\dots$ &&$3.65679089534028\dots$\\
%$-40$  && $1.266970658583983118836\dots$ &&$4.16362898337595\dots$\\
$\vdots$ && $\vdots$ && $\vdots$\\
$-136$ && $-4.839267599344343782986\dots$ &&$5.73824076491330\dots$\\
$-139$ && $-6                      $ &&$ 0$ \\
$-151$ && $-0.831356881792676920466\dots$ &&$6.69750855158616\dots$\\
$\vdots$ && $\vdots$ && $\vdots$\\
$-815$ && $121.9441031209309205888\dots$ &&$4.74925836934506\dots$\\
$-823$ && $312                     $ &&$ 0$ \\
$-824$ && $-322.9986066040975056735\dots$ &&$17.5028741140542\dots$\\
\hline
\end{tabular}
\end{table}

Observe that the numerics suggest that the holomorphic part of $3f_1-f_{12}$ has integral coefficients. This harmonic weak Maass form is mapped to a non-zero multiple of the theta function $g_0$ under $\xi_{1/2}$. So it should be viewed as an analogue of the function $\tilde H$ in the example of Section \ref{sect:8.2}.

%Notice that one can use this formula to compute the exponents in the
%product expansion and therefore some coefficients of the holomorphic
%part of $f_2$.

\end{document}